\documentclass[reqno,12pt,letterpaper]{amsart}
\usepackage{amsmath,amssymb,amsthm,graphicx,mathrsfs,url}
\usepackage[usenames,dvipsnames]{color}
\usepackage[colorlinks=true,linkcolor=Red,citecolor=Green]{hyperref}
\usepackage{amsxtra}
\usepackage{subcaption}
\usepackage{graphicx}
\usepackage{bbm}

\def\arXiv#1{\href{http://arxiv.org/abs/#1}{arXiv:#1}}

\def\?[#1]{\textbf{[#1]}\marginpar{\Large{\textbf{??}}}}
\def\smallsection#1{\smallskip\noindent\textbf{#1}.}
\let\epsilon=\varepsilon 

\newtheorem{theo}{Theorem}
\newtheorem{prop}{Proposition}[section]	
\newtheorem{defi}[prop]{Definition}

\newtheorem{ass}{Assumption}
\newtheorem{lemm}[prop]{Lemma}
\newtheorem{corr}[prop]{Corollary}
\newtheorem{rem}{Remark}

\newtheorem{disc}{Discussion}

\numberwithin{equation}{section}

\let\Re=\Real

\def\indic{\operatorname{1\hskip-2.75pt\relax l}}
\usepackage{scalerel}

\newcommand\reallywidehat[1]{\arraycolsep=0pt\relax%
\begin{array}{c}
\stretchto{
  \scaleto{
    \scalerel*[\widthof{\ensuremath{#1}}]{\kern-.5pt\bigwedge\kern-.5pt}
    {\rule[-\textheight/2]{1ex}{\textheight}} 
  }{\textheight} %
}{0.5ex}\\           
#1\\                 
\rule{-1ex}{0ex}
\end{array}
}

\title[Model Order Reduction for Delay Equations]{Model Order Reduction for (Stochastic-) Delay Equations With Error Bounds}

\author{Simon Becker and Lorenz Richter}
\email{simon.becker@damtp.cam.ac.uk}
\email{lorenz.richter@fu-berlin.de}
\address{DAMTP, University of Cambridge, Wilberforce Road, Cambridge CB3 0WA, UK}
\address{Institute of Mathematics, Freie Universität Berlin, 14195 Berlin, Germany}
\address{Institute of Mathematics, BTU Cottbus-Senftenberg, 03046 Cottbus, Germany}

\begin{document}

\begin{abstract}
We analyze a structure-preserving model order reduction technique for delay and stochastic delay equations based on the balanced truncation method and provide a system theoretic interpretation.
Transferring the framework of \cite{BH18}, we find error estimates for the difference between the dynamics of the full and reduced model. This analysis also yields new error bounds for bilinear systems and stochastic systems with multiplicative noise and non-zero initial states.
\end{abstract}

\maketitle


\section{Introduction}

In this article we study a delay-structure preserving model order reduction method, first discussed for finite-dimensional bilinear delay systems in \cite{GDBA}, based on the bilinear balanced truncation technique, for deterministic delay systems of the following types\footnote{our analysis immediately extends to $ N \varphi^{\operatorname{del}}(t-\tau)$ replaced by the sum $\sum_{i=1}^n N_i \varphi^{\operatorname{del}}(t-\tau_i).$}
 \begin{subequations}
 \label{eq:standardform}
\begin{eqnarray}
&\label{eq:delay1}\varphi^{\operatorname{del}'}(t) &= A \varphi^{\operatorname{del}}(t) +  N \varphi^{\operatorname{del}}(t-\tau)+ Bu(t), \text{ for } t \in (0,T)\\
&\varphi^{\operatorname{del}}(0)&=\varphi_0, \  \varphi^{\operatorname{del}}(t)=f(t) \text{ for } t \in (-\tau,0)  \quad \text{ and } \nonumber \\
&\label{eq:delay2}\varphi^{\operatorname{bild}'}(t) &= A \varphi^{\operatorname{bild}}(t) +  N \varphi^{\operatorname{bild}}(t-\tau)v(t)+ Bu(t), \text{ for }t \in (0,T)\\
&\varphi^{\operatorname{del}}(0)&=\varphi_0, \  \varphi^{\operatorname{del}}(t)=f(t) \text{ for } t \in (-\tau,0).\nonumber
\end{eqnarray}
\end{subequations}
on arbitrary (separable) Hilbert spaces $K$ for time-dependent control functions $u,v \in L^2((0,\infty))$ and a delay parameter $\tau>0$. In particular, for zero delay, $\tau=0$, \eqref{eq:delay2} reduces to the form of a standard \emph{bilinear system} and thus our error bound provides also a new error bound for the important class of bilinear systems with non-zero initial conditions, extending the analysis of \cite{BH18}. We also discuss different types of discrete and continuous delay, cf. \eqref{eq:typdelay}.

\medskip

Furthermore, we adapt our analysis to stochastic differential equations with delay (SDDE) on finite-dimensional system spaces $K \simeq \mathbb R^d$
\begin{equation}
\begin{split}
\label{eq:sdde}
dX_t &= (A X_t +Bu(t)) \ dt  + \sum_{i=1}^k N_i X_{t-\tau_i}\ dW_t^i, \qquad Y_t =CX_t \\
X_0&=\xi, \ X_t=f_t \text{ for } t \in (-\tau,0),
\end{split}
\end{equation}
where $(W_t^{i})_{i=1,\dots,k}$ are i.i.d.\@ copies of standard Brownian motion.

\medskip

\emph{Balanced truncation} is a well-established model order reduction technique, especially for linear (\eqref{eq:delay1} with $N=0$ \cite{CGP,BHRR,RS}), bilinear (\eqref{eq:delay1} with $\tau=0$ and $v=u$ \cite{BH18,R1,ZL,BD}), and stochastic differential equations with multiplicative noise (\eqref{eq:sdde} with $\tau_i=0$ \cite{BR15,BH18}). For bilinear with multiplicative noise, error bounds have only been obtained for zero initial condition so far - which we aim to overcome with this article, too. For stochastic systems with multiplicative noise, this has been only addressed in the recent work \cite{BHRR}. The bilinear balanced truncation method identifies, as has been rigorously proven at least for linear systems, a subspace of jointly reachable and controllable states and aims to preserve these subspaces well under model order reduction \cite{CGP,BD,BR15,R1}. The two properties are approximately captured by positive-definite operators that are called \emph{Gramians}. The dominant eigenspaces of the product of two Gramians are then used to define the reduced order model. 
The method we consider is inspired by linear, bilinear, and stochastic balanced truncation theory and preserves the delay-structure of the original system. 

In this article we study delay equations \eqref{eq:delay1}, \eqref{eq:delay2} using bilinear balanced truncation, and the SDDE \eqref{eq:sdde} using stochastic balanced truncation. We show that the Gramians still have a system theoretic interpretation and also derive error bounds for the reduced order model. 

We also emphasize that balanced truncation is a method particularly designed for dissipative systems. For all balanced truncation methods, one therefore assumes that the operator $A$ generates an exponentially stable and strongly continuous semigroup. Moreover, \eqref{eq:delay1} is less well-adapted to the dissipative structure, since -- without further assumptions on $N$ -- the delay term can create growing modes in the dynamics. This is prevented in \eqref{eq:delay2} by assuming that the control $v$ is \emph{small}.

To take non-zero initial states into account, we consider for equations \eqref{eq:standardform} and \eqref{eq:sdde} a space $Y:=\operatorname{span}\left\{\varphi_i; \ i \in \{1,\dots,N\} \right\}$ of admissible input states, where $(\varphi_i)$ is an orthonormal system in $K$ and introduce a map $B_{\text{in}}u = \sum_{i=1}^n u_i \varphi_i$ that we include in the model order reduction.

Before stating the main results of this article, let us mention other model order reduction techniques for delay systems. By rewriting the delay equation as a linear equation, using the head-tail representation \cite{CZ} on an infinite-dimensional space, a reduction method based on linear balanced truncation theory has been proposed in \cite{JDM12}. Apart from that article, applying the method of balanced truncation to delay equations, other model order reduction techniques such as rational approximation methods \cite{19,20}, interpolation methods \cite{4}, Krylov space methods \cite{10}, moment matching based methods \cite{24,SA} have been proposed.
\medskip 

Many of the above methods, however, do not fully preserve the delay structure, which is fixed by the method proposed in this article.

\subsection{Outline of the article}
\begin{itemize}
\item In Section \ref{sec:BTinNut} we discuss the algorithmic aspects of the balanced truncation method that we propose for delay systems. 

\item Section \ref{sec:DDE} provides an overview over the $C_0$-semigroup approach to deterministic delay equations. 

\item In Section \ref{sec:BBTFDS}, we introduce the balanced truncation theory framework including the central objects of the theory, i.e. the bilinear Gramians and the Hankel operator for the deterministic equations. We then show in Propositions \ref{homsys} and \ref{homsys2} that the bilinear Gramians preserve the structure of the equation also for delay equations.

\item The proof of the error bounds for bilinear and deterministic delay equations with non-zero initial data, stated in Theorem \ref{theo:err}, is given in Section \ref{sec:Proof1}.

\item In Section \ref{sec:SDDE} we then treat the stochastic delay systems. 

\item Finally, in Section \ref{sec:examples}, we consider delay models from physics including Stuart-Landau oscillators and generalized Langevin equations, and apply the balanced truncation method for delay systems.

\end{itemize}
\subsection{Main results}

To state the error estimates, let $\left\lVert N \right\rVert, \left\lVert C \right\rVert, $ and $\left\lVert B \right\rVert$ denote the maximum of the respective norm for the full and reduced system and $M,\omega>0$ such that for both systems the corresponding semigroup satisfies $\left\lVert T(t) \right\rVert \le Me^{-\omega t}.$ We also introduce the norm
\[\Vert u \Vert_{L^{2 \vee \infty}} :=\operatorname{max}\left\{\left\lVert u \right\rVert_{L^{\infty}((0,t),\mathbb R^n)}, \left\lVert u \right\rVert_{L^{2}((0,t),\mathbb R^n)} \right\}.\]

Our main error estimate for deterministic systems is stated in the following theorem and applies to bilinear systems as well:

\begin{theo}[Error bound deterministic systems]
\label{theo:err}
Let $\mathcal{H}\simeq \mathbb{R}^m$ be the output space and consider the difference of two solutions to \eqref{eq:delay2}. We assume that both solutions satisfy the stability condition $M\left\lVert N \right\rVert /\sqrt{2\omega}<1$ and $\Vert v \Vert_{L^2(0,T)} \le 1$, such that the Volterra series converges \cite[Lemma $A.1$]{BH18}. Then, for control functions $u,v \in H^1((0,T),\mathbb R)$, initial states $\varphi_0 = \sum_{i=1}^k \langle \textbf{w} , \widehat{e_i} \rangle \phi_i$ and $\widetilde{\varphi}_0:= \sum_{i=1}^k \langle  \textbf{w} , \widehat{e_i} \rangle \widetilde{\phi}_i$, and zero history function, it follows for $H^{\operatorname{bil}}$ being the Hankel operator associated to the system that $\Delta(C\varphi^{\operatorname{bild}})$, which is the difference of the output of the full and reduced system, satisfies  
\begin{equation}
\begin{split}
\label{eq:erresm2}
\left\lVert \Delta(C\varphi^{\operatorname{bild}}) \right\rVert_{L^2((0,\infty), \mathbb R^m)} &\le 4   \left\lVert \Delta(H^{\operatorname{bil}})\right\rVert_{\operatorname{TC}} \Bigg(\left\lVert \varphi_0 \right\rVert_{K} \operatorname{max} \left\{ 1, \left\lVert v \right\rVert_{L^{\infty}(0,T)} \right\} \\
&\qquad +\operatorname{max} \left\{\Vert u \Vert_{L^{2}(0,T)}, \left\lVert v \right\rVert_{L^1(0,T)} \right\}\left\lVert u \right\rVert_{L^{\infty}(0,T)} \Bigg).
\end{split}
\end{equation}
\end{theo}
The proof of this result is given in Section \ref{sec:Proof1}.

\medskip
The corresponding error bound for stochastic (delay) differential equations with multiplicative noise is stated in the following theorem: 
\begin{theo}[SDDE error bound]
\label{theo:sdde}
For control functions $u \in L^{\infty}_{\omega}L^2_t(\Omega_{(0,T)},\mathbb R^n)$, zero history functions, i.e. $f_t\equiv 0$, and two solutions to \eqref{eq:sdde} with initial conditions $\xi := \sum_{i=1}^k \langle \textbf{v} , \widehat{e_i} \rangle \xi_i$ with $L^2(\Omega,\mathcal F_0,K)$ orthonormal system $(\xi_i),$ and $\widetilde{\xi}:= \sum_{i=1}^k \langle  \textbf{v} , \widehat{e_i} \rangle \widetilde{\xi}_i,$ it follows that for two independent Wiener processes, where both solutions satisfy $M\left\lVert N \right\rVert /\sqrt{2\omega}<1,$ and  $H^{\operatorname{sdde}}$ being the Hankel operator associated to the system
\begin{equation*}
\begin{split}
&\left\lVert \Delta \left(CX^{\operatorname{sdde}}\right) \right\rVert_{L^2(\Omega_{(0,T)},\mathbb R^m)} \ \le  \left\lVert \Delta(H^{\operatorname{sdde}}) \right\rVert_{\operatorname{TC}} \left(\left\lVert \xi \right\rVert_{L^{2}(\Omega;K)} + 2 \left\lVert u \right\rVert_{L^{\infty}_{\omega}L^2_t(\Omega_{(0,T)},\mathbb R^n)} \right) 
\end{split}
\end{equation*}
\end{theo}

The proof of this result is given at the end of Section \ref{sec:SDDE}.
\subsection*{Notation}
The space of bounded linear operators between Banach spaces $X,Y$ is denoted by $\mathcal L(X,Y)$ and just by $\mathcal L(X)$ if $X=Y.$ The operator norm of a bounded operator $T \in \mathcal L(X,Y)$ is written as $\left\lVert T \right\rVert$. The space of trace class operators is denoted by $\operatorname{TC}(X,Y)$ and the space of Hilbert-Schmidt operators by $\operatorname{HS}(X,Y).$
 In particular, we recall that for a linear operator $T \in \operatorname{TC}(X,Y)$, where $X$ and $Y$ are now separable Hilbert spaces, the trace norm is given as
\begin{equation}
\label{tracenorm}
\left\lVert T \right\rVert_{\operatorname{TC}}=\sup \left\{ \sum_{n \in \mathbb{N}} \left\lvert \langle f_n,T e_n  \rangle_Y \right\rvert : (e_n) \text{ ONB of } X \text{ and } (f_n) \text{ ONB of } Y \right\}.
\end{equation}

The resolvent set of an operator $A$ is denoted by $\rho(A)$ and we say that $g=\mathcal O (f)$ if there is a $C>0$ such that $\left\lVert g \right\rVert \le C \left\lVert f \right\rVert.$ In order not to specify the constant $C$, we also write $\left\lVert g \right\rVert \lesssim  \left\lVert f \right\rVert.$
The domain of unbounded operators $A$ is denoted by $D(A).$

Let $\mathnormal H$ be a separable Hilbert space. To define the Hankel operator we require a decomposition of the positive Gramians. For this purpose, we introduce the Fock space $F(\mathnormal H)$ of $\mathnormal H$-valued functions $F(\mathnormal H):=\bigoplus_{k=0}^{\infty}F_k(\mathnormal H)$ where $F_k(\mathnormal H):=L^2((0,\infty)^k,\mathnormal H)$ and $F_0(\mathnormal H):=\mathnormal H.$ Thus, elements of the Fock space are sequences of $F_k$-valued elements.

We introduce function spaces $L^1_iL^2_{k-1}$ and $\mathscr H^{\infty}_i \mathscr H^{2}_{k-1}$ norms which for $\mathnormal H$-valued functions functions $f:(0,\infty)^k \rightarrow  \mathnormal H$ and $g: \mathbb C_{+}^k \rightarrow \mathnormal H$ are defined by
\begin{equation}
\label{eq: mix}
\begin{split}
&\left\lVert f \right\rVert_{L^1_iL^2_{k-1}(\mathnormal H)} = \int_0^{\infty} \left\lVert f(\bullet,\dots,\bullet,s_i, \bullet,\dots,\bullet) \right\rVert_{L^2((0,\infty)^{k-1},\mathnormal H)} \ ds_i \text{ and } \\
&\left\lVert g \right\rVert_{\mathscr H^{\infty}_i \mathscr H^{2}_{k-1}(\mathnormal H)}= \sup_{s_i \in \mathbb C_{+}} \left\lVert g(\bullet,\dots,\bullet,s_i, \bullet,\dots,\bullet) \right\rVert_{\mathscr H^2((0,\infty)^{k-1},\mathnormal H)}. 
\end{split}
\end{equation}
For $k$-variable functions $h$ we sometimes also write $h^{(i)}(s,t):=h(s_1,\dots,s_{i-1},t,s_{i},\dots,s_{k-1})$ in order to shorten the notation. We denote by $H^1$ the Sobolev space of $L^2$ functions whose first weak derivative is in $L^2$ as well.

\section{Balanced truncation for delay systems in a nutshell}
\label{sec:BTinNut}
In this section, we provide a brief overview over the model order reduction method studied in this article and the computability of the error bounds we stated in Theorems \ref{theo:err} and \ref{theo:sdde}.

To obtain the reduced order model for \emph{delay systems on finite-dimensional system spaces} $K$ of the form \eqref{eq:standardform}, or \eqref{eq:sdde} with (possibly) multiple delays, we compute positive (semi-)definite observability $\mathscr O$ and reachability $\mathscr P$ Gramians from the following Lyapunov equations
\begin{equation}
\begin{split}
\label{eq:Lyapunoveq}
&A\mathscr P + \mathscr PA^T + \sum_{i=1}^k N_i \mathscr P N_i^T + BB^T + B_{\operatorname{in}}B_{\operatorname{in}}^T =0 \text{ and }\\
&A^T\mathscr O + \mathscr O A + \sum_{i=1}^k N_i^T \mathscr O N_i + C^TC = 0.
\end{split}
\end{equation}
Since both $\mathscr O$ and $\mathscr P$ are positive semidefinite, they can be decomposed as $\mathscr O = W^*W$ and $\mathscr P =RR^*.$
Let $\mathscr O$ and $\mathscr P$ have both full rank for simplicity, then the \emph{balanced representation} is obtained by performing a singular value decomposition (SVD) $WR = V \Sigma U^*$ and introducing operators $Q:=\Sigma^{-1/2} V^*W$ and $Q^{-1}:=RU\Sigma^{-1/2}$ such that the new balanced matrices are given by
\begin{equation}
\begin{split}
A_b := QAQ^{-1}, \ N_b:=QNQ^{-1}, \ B_b:=QB, \ B_{\operatorname{in} b}:=QB_{\operatorname{in}},  \text{ and }C_b:=CQ^{-1}.
\end{split}
\end{equation}
To obtain a \emph{reduced system}, the smallest singular values of the matrix $\Sigma$ are discarded. 
The error bounds stated in this article are given in terms of the trace distance of the difference of certain Hankel operators $H$, introduced in Def. \ref{HSop}, which we denote by $\Delta(H),$ for the full and reduced system.

To actually compute the singular values of $\Delta(H)$, and thus the trace norm of $\Delta(H)$, one defines an \emph{error system} 
\begin{equation}
\begin{split}
\label{eq:compsys}
&\widehat A:=\left(\begin{matrix} A & 0 \\ 0 & \tilde{A} \end{matrix} \right), \ \widehat B := (B,\tilde{B})^{T}, \ \widehat B_{\operatorname{in}} := (B_{\text{in}},\tilde{B}_{\text{in}})^T,\widehat C:=(C,-\tilde C), \text{ and } \widehat N:=\left(\begin{matrix}N & 0 \\ 0& \tilde N \end{matrix} \right),
\end{split}
\end{equation}
where operators without tilde belong to the original system in \eqref{eq:standardform}, which we call \emph{System $1$}  and operators with tilde correspond to a second system, that we call \emph{System $2$}. This could be any other system with the same structure as the reduced system.
Then one can define Gramians $\widehat{\mathscr O}=\widehat W^{*}\widehat W$ and $\widehat {\mathscr P }=\widehat R \widehat R^{*}$ of this error system \eqref{eq:compsys} that also satisfy Lyapunov equations
\begin{equation}
\begin{split}
&\widehat A \widehat{ \mathscr P} + \widehat{ \mathscr P}\widehat A^T + \sum_{i=1}^k \widehat N_i \widehat{ \mathscr P} \widehat N_i^T + \widehat B\widehat B^T + \widehat B_{\operatorname{in}}\widehat B_{\operatorname{in}}^T =0 \text{ and }\\
&\widehat A^T \widehat{ \mathscr O} + \widehat{ \mathscr O} \widehat A + \sum_{i=1}^k \widehat N_i^T \widehat{ \mathscr O} \widehat N_i + \widehat C^T\widehat C = 0.
\end{split}
\end{equation}
 We can then perform a singular value decomposition $\widehat W\widehat R = \widehat V \Lambda \widehat U^{*}$ with diagonal operator $\Lambda$ that contains all singular values of the error system \eqref{eq:compsys} on its diagonal \cite[Theorem $5.1$]{RS}. Hence, we find $\left\lVert \Delta(H) \right\rVert_{\operatorname{TC}} = \sum_{\lambda \in \Lambda} \lambda,$ as there exist unitary mappings \cite[Prop. $6.1$]{RS} $U: \overline{ \operatorname{ran}}(\widehat W\widehat R) \rightarrow \overline{ \operatorname{ran}}(\widehat H)$ and $V: \operatorname{ker}^{\perp}(\widehat W\widehat R) \rightarrow\operatorname{ker}^{\perp}(\widehat H)$
such that 
\[ \Delta(H) \vert_{\operatorname{ker}^{\perp}(\widehat H)} = \widehat H \vert_{\operatorname{ker}^{\perp}(\widehat H)}=U\left(\widehat W \widehat  R\right) \vert_{\operatorname{ker}^{\perp}(\widehat  W\widehat  R)} V^{*} \vert_{\operatorname{ker}^{\perp}(\widehat H)}. \]

\section{Deterministic Delay equations}
\label{sec:DDE}
For our analysis of delay equations, we start with a $C_0$-semigroup approach to study well-posedness, Volterra kernel expansions, that will be essential in the proof of the error bounds, and stability properties. 
\label{sec:Delay}
\subsection{Delay equations in a semigroup framework}

We start by introducing a \emph{delay operator} $\Phi \in \mathcal L (H^{1}((-r,0),K),K)$, where $H^1$ is the standard Sobolev space. By normalizing appropriately, we can always assume that the history interval, i.e.\@ the interval by which the dynamics reaches back in time, is $(-r,0)$ for some fixed $r>0.$ To fully define the dynamics, we require an initial value $\varphi_0 \in K$ and a history function $ f\in L^2((-r,0),K).$

Then, for a control function $u \in L^2((0,\infty), \mathbb R^n)$ and bounded operators $B \in \mathcal L(\mathbb{R}^n, K)$ of the form $Bu = \sum_{i=1}^n \psi_i u_i$ for some $\psi_i \in K,$ we study delay equations of type \eqref{eq:delay1}
\begin{equation}
\begin{split}
\label{eq:delay}
\varphi^{\operatorname{del}'}(t) &= A\varphi^{\operatorname{del}}(t)+ \Phi\left( \varphi^{\operatorname{del}}\right)(t) + Bu(t), \ t>0 \\
\varphi^{\operatorname{del}}(0)&= \varphi_0, \quad \varphi^{\operatorname{del}}(\sigma)= f(\sigma), \text{ for } \sigma \in (-r,0).
\end{split}
\end{equation} 

We consider two types of delays operators $\Phi: H^{1}((-r,0),K) \rightarrow K$ for our error analysis. 
\begin{ass}[Delay operators] 
\label{ass:delays}
For bounded operators $N \in \mathcal L(K)$, and two types of delays \begin{itemize}
    \item constant time delays $r>\tau>0$ or
    \item integral delays $g \in L^1(-r,0)$ such that $\int_{-r}^0 \vert g(s) \vert \ ds \le 1$,
\end{itemize}
we introduce delay operators $\Phi_d$ and $\Phi_c$ for $f \in H^{1}((-r,0),K)$
\begin{equation}
\label{eq:typdelay}
 \Phi_df  :=  Nf(-\tau) \text{ or } \Phi_cf:= \int_{-r}^0 Nf(s) g(s) \ ds.
 \end{equation}
 \end{ass}
 
\bigskip

Furthermore, we introduce a vectorized control-to-state map 
(\emph{control operator}) $\mathbf{B}$, (\emph{initial-state operator}) $\mathbf{B}_{\operatorname{in}}$, and a state-to output map (\emph{observation operator}) $\mathbf{C}$ such that on $\mathcal K= K \times L^2((-r,0), K)$
\begin{equation}
\begin{split}
\label{eq:definitions}
&\mathbf{B}:=(B,0)  \in \mathcal L(\mathbb R^n,\mathcal K), \quad  \mathbf{B}_{\operatorname{in}}:=(B_{\text{in}},0)    \in \mathcal L(\mathbb R^k,\mathcal K),\\
&\text{ and }\mathbf{C}(\varphi,f):=C\varphi \text{ with }\mathbf{C} \in \mathcal L(\mathcal K,\mathcal H).
\end{split}
\end{equation}

To apply the theory of $C_0$-semigroups, we introduce operators $\mathcal A_0$ and $ \mathcal N$ on $\mathcal K$
\begin{equation}
\label{eq:generator}
 \mathcal A_0:=\left(\begin{matrix} A  & 0 \\ 0 & \frac{d}{d \sigma} \end{matrix} \right) \text{ and } \mathcal N:=\left(\begin{matrix} 0  & \Phi \\ 0 & 0 \end{matrix} \right) \in \mathcal L(D(\mathcal A_0),\mathcal K) 
 \end{equation}
with domain $D(\mathcal A_0):=\left\{ (\varphi,f) \in D(A) \times H^{1}((-r,0), K): f(0)=\varphi \right\} \subset \mathcal K$\footnote{The linear space $D(\mathcal A_0)$ carries the graph norm $\left\lVert x \right\rVert_{D(\mathcal A_0)} = \left\lVert \mathcal A_0 x \right\rVert+ \left\lVert x \right\rVert.$} where $A$ is itself a generator of a strongly continuous semigroup $(T(t))$ on $K$.  
We then define an operator 
\[S_t: K \rightarrow L^2((-r,0),K)\text{ by }(S_t\varphi)(-\tau):=\indic_{[0,\infty]}(t-\tau) T(t-\tau)\varphi\]
and $T_{\leftarrow}(t) \in \mathcal L(L^2((-r,0),K))$ as the nilpotent left-shift semigroup 
\[(T_{\leftarrow}(t)\varphi)(-\tau)=(\indic_{[-r,0]}\varphi)(t-\tau).\]
It can be shown \cite[Theorem $3.25$]{BP} that the operator $\mathcal A_0$ is the generator of a $C_0$-semigroup 
\begin{equation}
\label{eq:T0}
 \mathcal T_0(t):= \left( \begin{matrix} T(t) & 0 \\ S_t & T_{\leftarrow}(t) \end{matrix} \right).
 \end{equation}
 We continue with a stability condition that plays the analogous role to exponential stability of the semigroup for delay equations: 
 \begin{ass}[MV-condition]
We assume the semigroup $\mathcal T_0$ in \eqref{eq:T0} and operator $\mathcal N$ to satisfy a \emph{Miyadera-Voigt $L^2$-condition} for some $q<1$, i.e.
\begin{equation}
\label{eq:MV2}
 \int_0^{\infty} \left\lVert \mathcal N \mathcal T_{0}(r) \varphi \right\rVert^2 \ dr \le q \left\lVert \varphi \right\rVert^2.
 \end{equation}
 \end{ass}

\begin{defi}
\label{VM2}
We say that $\mathcal T_0$ and $\mathcal N$ satisfy a (truncated) \emph{Miyadera-Voigt $L^1$-condition} if for some $q<1$ and $t_0>0:$
\begin{equation}
\label{eq:MV1}
 \int_0^{t_0} \left\lVert \mathcal N \mathcal T_{0}(r) \varphi \right\rVert \ dr \le q \left\lVert \varphi \right\rVert.
 \end{equation}
By H\"older's inequality, \eqref{eq:MV2} implies \eqref{eq:MV1} for some $t_0>0.$
\end{defi}

Under Assumption \ref{ass:delays}, the Miyadera-Voigt perturbation theorem \cite[Corollary $3.16$]{EN} implies that the state operator $\mathcal A:=\mathcal A_0+\mathcal N$ is the generator of a semigroup $(\mathcal T(t) )$ and the delay equation \eqref{eq:delay} is well-posed \cite[Theorem $3.26$]{BP}. That is, the solution to \eqref{eq:delay} is continuous and we can write the solution to \eqref{eq:delay} by Duhamel's formula for $x_0=(\varphi_0,f) \in \mathcal K$ and $u \in L^2((0,\infty), \mathbb R^n)$ as
\begin{equation}
\begin{split}
\label{eq:sol}
Z(t)&=\mathcal T(t)x_0+ \int_0^t \mathcal T(t-s) \mathbf{B}u(s) \ ds \text{ with output }Y(t)=\mathbf{C}(Z(t)).
\end{split}
\end{equation}
Let $\pi_1$ be the projection from $\mathcal K\ni (x,f) \mapsto x$, then $\pi_1(Z)$ solves \eqref{eq:delay}.

\begin{rem}
Many previous ideas related to model order reduction methods for delay utilized the linear structure of the representation \eqref{eq:sol}. The computational issue with this approach is that the system is inherently infinite-dimensional.

On the other hand, when studying delay system using linear balanced truncation, there are explicit criteria for exponential stability of the semigroup $(\mathcal T(t))$:
\begin{lemm}\cite[Corollary $4.6$]{BFS}
\label{lemm:stab}
If the semigroup $(T(t))$, generated by $A$ in the delay equation \eqref{eq:delay}, satisfies $\left\lVert T(t) \right\rVert \le M e^{-\omega t} $ with $\omega$ strictly larger than some $ \alpha \ge 0$ and $ \frac{Me^{\alpha \tau}}{\omega-\alpha} \left\lVert N \right\rVert<1,$
then the decay bound $\omega_0(\mathcal A)$ of the delayed semigroup $\mathcal T(t)$ satisfies $\omega_0(\mathcal A)>\alpha \ge 0$, i.e. there is some $\mathcal M>0$ such that $\left\lVert \mathcal T(t) \right\rVert \le \mathcal M e^{-\omega_0(\mathcal A) t}.$
\end{lemm}
\end{rem}

\subsection{Volterra series expansion of the dynamics}

We record that Definition \eqref{eq:MV2} implies the existence of an operator $\mathcal S \in \mathcal L( \mathcal K, L^2((0,\infty),\mathcal K))$ extending $\mathcal N \mathcal T_{0}(t)$ from $D(\mathcal A_0)$ to all of $\mathcal K$ with norm $\left\lVert \mathcal S \right\rVert \le \gamma<1.$ The operator $\mathcal S$, that extends $\mathcal N \mathcal T_{0}(t)$, satisfies $\mathcal S(s)\mathcal T_0(t)= \mathcal S(t+s)$. Let $D(\mathcal A_0) \ni x_n \rightarrow x \in \mathcal K$, then 
\[ \mathcal S(s)\mathcal T_0(t)x= \lim_{n \rightarrow \infty} \mathcal S(s)\mathcal T_0(t)x_n = \lim_{n \rightarrow \infty} \mathcal N \mathcal T_0(s)\mathcal T_0(t)x_n =\mathcal S(s+t)x. \]

The Miyadera-Voigt perturbation theorem \cite[Ch.$3$, Sec.$3$]{EN} allows us to express the semigroup $(\mathcal T(t))$ in $\mathcal L(K)$ generated by the state operator $\mathcal A=\mathcal A_0+\mathcal N$ as a series $\mathcal T(t)x = \sum_{i=0}^{\infty} \mathcal V^i  \mathcal T_0(t)x.$

The operators $\mathcal V^i$ are the so-called \emph{Volterra operators} $\mathcal V \in \mathcal L(L^p([0,T],\mathcal K))$ defined, for any $p \in [1,\infty]$, as 
\begin{equation}
\label{eq:young}
 (\mathcal VF)(t)x = \int_0^t  F(t-s)  (\mathcal Sx)(s) \ ds = ( F *(\mathcal Sx  \indic_{[0,\bullet] }))(t), \text{ for all }t  \in [0,1] 
 \end{equation}
such that by Young's inequality $\Vert \mathcal V F \Vert_{L^p} \le \Vert \mathcal S \Vert_{\mathcal L(L^1,\mathcal K)} \Vert F \Vert_{L^p}.$

From \eqref{eq:sol} and Fubini's theorem it follows that for sets with $i \in \mathbb N$ \\ $\Delta_i(t):=\left\{s \in \mathbb R^i; 0\le s_i\le\dots\le s_1\le t \right\}$ and \emph{delay Volterra kernels}  
\begin{equation}
\begin{split}
\label{eq:Voltker}
h^{\operatorname{delay}}_i(t)&:=\mathcal O_{i-1}(t)\mathbf B \in \mathcal L(\mathbb R^n,\mathcal H) \text{ and } h^{\operatorname{delay}}_{i,\operatorname{in}}(t):=\mathcal O_{i-1}(t)\mathbf B_{\operatorname{in}} \in \mathcal L(\mathbb R^k,\mathcal H)\\
&\text{ where }\mathcal O_{i}(t=(t_1,\dots,t_{i+1}))y:=\mathcal C \mathcal T(t_1)  \prod_{l=2}^{i+1} \left(\mathcal S(t_l) \right)y
\end{split}
\end{equation}
the solution $Y$, in \eqref{eq:sol}, for initial conditions $x_0:=(\varphi_0,0)$, and $u \in L^p((0,\infty),\mathbb R^n)$ is a function $Y \in L^p_{loc}((0,\infty), \mathcal H)$ given by $Y(t) = K_1(t)+K_2(t)$ where
\begin{equation}
\begin{split}
\label{eq:S}
K_1(t)= \sum_{i=0}^{\infty} \mathcal C\mathcal V^i  \mathcal T_0(t)x \quad\text{ and }\quad K_2(t)&=  \int_0^t \mathbf C \mathcal T(t-s)\mathbf Bu(s) \ ds\\
&=  \sum_{i=1}^{\infty} \int_{\Delta_i(t)} h^{\operatorname{delay}}_i(t-s_1,\dots,s_{i-1}-s_i) \ u(s_i)\ ds
\end{split}
\end{equation}
where $\mathcal C$ is defined in \eqref{eq:definitions}.
In \eqref{eq:Voltker} and \eqref{eq:S}, we introduced the delay Volterra kernels, which we shall write down more explicitly, for delay types \eqref{eq:typdelay} indicated by indices $\textbf{d}$(iscrete)$\vert\textbf{c}$(ontinuous), in the following Lemma:
\begin{lemm}
\label{kernels1}
The delay Volterra kernels \eqref{eq:Voltker} satisfy for any $j \in \mathbb N$
\begin{equation}
\begin{split}
\label{eq:discdel}
h^{\operatorname{delay}}_{d,j\vert d,j, \operatorname{in}}(t_1,\dots,t_j)&:= CT(t_1)\prod_{i=2}^{j}  \left(\indic_{(0,\infty)}(t_i-\tau)N T(t_{i}-\tau)\right)B_{\vert \operatorname{in}} \\
h^{\operatorname{delay}}_{c,j \vert c,j, \operatorname{in}}(t_1,\dots,t_j)&:= CT(t_1)\prod_{i=2}^{j}  \left(\int_{-r}^0 \indic_{(0,\infty)}(t_i+s)N T(t_{i}+s) g(s) ds \right)B_{\vert \operatorname{in}}.
\end{split}
\end{equation}
\end{lemm} 
\begin{proof}
We restrict our proof to discrete delays $\Phi=\Phi_d$ in the proof. Consider $(\varphi,f) \in D(\mathcal A_0),$ then an explicit computation shows that
\[\mathcal S(t)(\varphi,f)= \Phi(S_t \varphi +T_{\leftarrow}(t) f) =\indic_{(0,\infty)}(t-\tau) N T(t-\tau)\varphi + \indic_{[0,\tau]}(t)  N f(t-\tau).\]
Let $(\varphi_i,f_i) \in D(\mathcal A_0)$ converge to $(\varphi,f) \in \mathcal K,$ then by continuity of $\mathcal S(t)$
\begin{equation}
\begin{split}
 \mathcal S(t)(\varphi,f) &= \lim_{i \rightarrow \infty}\mathcal S(t)(\varphi_i,f_i)= \lim_{i \rightarrow \infty} \mathcal N \mathcal T_0(t)(\varphi_i,f_i)=\lim_{i \rightarrow \infty} (\Phi( S_t\varphi_i+ T_{\leftarrow}(t)f_i),0) \\
&= \lim_{i \rightarrow \infty} \left(\indic_{(\tau,\infty)}(t) N T(t-\tau) \varphi_i \ +\indic_{[0,\tau]}(t) N f_i(t-\tau), 0 \right) \\
&=     \left(\indic_{(0,\infty)}(t-\tau) N T(t-\tau) \varphi +  \indic_{[0,\tau]}(t) N f(t-\tau),0\right)
\end{split}
\end{equation}
shows that $\mathcal{S}(t)(\varphi,0)=\left( \indic_{(0,\infty)}(t-\tau) N T(t-\tau) \varphi,0 \right).$
From the definition of the Volterra kernels $h^{\operatorname{delay}}_m$ \eqref{eq:Voltker} we thus conclude that the first equation in \eqref{eq:discdel} holds. 
\end{proof}

\section{Bilinear balanced truncation for delay systems}
\label{sec:BBTFDS}
In this section, we provide the necessary tools from bilinear balanced truncation theory and apply it to delay systems.

Let $\mathcal{H}$ be a separable Hilbert space and $C \in \mathcal L(K,\mathcal{H})$ the observation operator, we then introduce the central object of the bilinear balanced truncation theory, the bilinear Gramians, cf. \cite{BH18,ZL}. In particular, the mapping and regularity properties have all been shown in \cite{BH18}.
\begin{defi}[Gramians]
\label{Gramians} 
Let $O_0(t_1):=CT(t_1)$. Then, for $i \in \mathbb{N}$ and $y \in K$ define $O_{i}(t_1,\dots,t_{i+1})y:=CT(t_1)  \prod_{j=2}^{i+1} \left(N T(t_j) \right)y$ and bounded operators $\mathscr{O}_i$ for $x,y \in K$ by $\langle x, \mathscr{O}_iy \rangle_K := \int_{(0,\infty)^{i+1}}\langle O_i(s)x, O_i(s)y \rangle_{\mathcal H} \ ds,$
which are summable in operator norm. The \emph{bilinear observability Gramian} $\mathscr{O}^{\operatorname{bil}} \in \mathcal L(K)$ is then given as $\mathscr{O}^{\operatorname{bil}} := \sum_{i=0}^{\infty} \mathscr{O}_i \in \mathcal L(K).$

For the \emph{bilinear reachability Gramian}, we define for $i \in \mathbb N$ and $y \in K$   \[P_{i}(t_1,\dots,t_{i+1})y:= \prod_{j=1}^{i}  \left(T(t_j)^*N^*\right)T(t_{i+1})^*y \]
such that
\begin{equation*}
\begin{split}
\label{eq:P0}
\langle x,\mathscr{P}_0y \rangle_K
&:=\int_{(0,\infty)} \left\langle T(s)^*x, BB^* T(s)^*y \right\rangle_{K } \ ds +\langle x,B_{\operatorname{in}}B_{\operatorname{in}}^* y \rangle \text{ and for }i \in \mathbb N,\\
\langle x,\mathscr{P}_{i}y \rangle_K
&:=\int_{(0,\infty)^{i+1}} \left\langle P_i(s)x, BB^*P_{i}(s)y \right\rangle_{K} \ ds + \int_{(0,\infty)^{i}} \left\langle P_{i-1}(s)x, B_{\operatorname{in}}B_{\operatorname{in}}^* P_{i-1}(s)y \right\rangle_{K} \ ds.
\end{split}
\end{equation*}
The \emph{bilinear reachability Gramian} is the operator $\mathscr P^{\operatorname{bil}}:= \sum_{i=0}^{\infty} \mathscr{P}_i \in \operatorname{TC}(K).$ 
\end{defi}
We introduce operators $W^{\operatorname{bil}}$ and $R^{\operatorname{bil}}$ such that the \emph{observability Gramian} is $\mathscr{O}^{\operatorname{bil}}= W^{\operatorname{bil}*}  W^{\operatorname{bil}}$ and the reachability Gramian is $\mathscr{P}^{\operatorname{bil}} =  R^{\operatorname{bil}}  R^{\operatorname{bil}*}.$ 
\begin{defi}[Observ. \& reach. map]
For $i \in \mathbb{N}_0$ we define the family $W_i\in \mathcal L\left( K, F_{i+1}\left(\mathcal{H}  \right)\right)$ of operators that map $K \ni x \mapsto O_i(\bullet)x$ such that $\left\lVert  W_i \right\rVert = \mathcal O\left( \left(M \left\lVert N \right\rVert ( 2 \omega )^{-1/2} \right)^i\right).$ Then, we can define the \emph{bilinear observability map} $W^{\operatorname{bil}} \in \mathcal L\left(K, F\left(\mathcal{H} \right)\right)$ by $W^{\operatorname{bil}}(x):=\left( W_i(x) \right)_{i \in \mathbb{N}_0}.$

Let $R_i \in \operatorname{HS}\left(F_{i+1}\left(\mathbb{R}^{n}\right) \oplus F_{i}(\mathbb R^k),K\right)$ be given by  
\begin{equation*}
\begin{split}
&R_0(f,g):= \int_{(0,\infty)} T(s) B f(s) \ ds+  B_{\operatorname{in}}g  \text{ and for } i \in \mathbb N \\
&R_i(f,g):=\int_{(0,\infty)^{i+1}} P_i(s)^*Bf(s) \ ds + \int_{(0,\infty)^{i}} P_{i-1}(s)^*B_{\operatorname{in}}g(s) \ ds.
\end{split}
\end{equation*}
The \emph{bilinear reachability map} is defined as $R^{\operatorname{bil}} \in \operatorname{HS}\left( F\left(\mathbb{R}^n\right) \oplus \bigoplus_{i=0}^{\infty}F_i(\mathbb R^k),K\right)$ such that $(f_i,g_i)_{i \in \mathbb N_0} \mapsto \sum_{i=0}^{\infty}  R_i( f_i,g_i).$
For subsequent use, we also define maps 
\begin{equation}
\label{eq:FandG}
 F_i(f_i):=R_i(f_i,0)\text{ and }G_i(g_i):=R_i(0,g_i).
 \end{equation}
\end{defi}
Using the above two operators $R^{\operatorname{bil}}$ and $W^{\operatorname{bil}}$, we can now introduce the bilinear Hankel operator, cf. \cite{BH18}.
\begin{defi}[Hankel operator]
\label{HSop}
The \emph{Hankel operator} is the Hilbert-Schmidt operator $H^{\operatorname{bil}}:= W^{\operatorname{bil}}  R^{\operatorname{bil}} \in \operatorname{HS}\left( F(\mathbb{R}^n)\oplus \bigoplus_{i=0}^{\infty}F_i(\mathbb R^k),F(\mathcal{H})\right)$.
\end{defi}
In particular, if $\mathcal H$ is finite-dimensional then $H^{\text{bil}}$ is of trace-class.

To relate the delayed dynamics to the bilinear Gramians we introduce the integral kernels of the bilinear Hankel operator:
\begin{defi}
\label{bilvolt}
The \emph{bilinear Volterra kernels} $h_{m}$ and $h_{m,\operatorname{in}}$ are the functions defined for $m \in \mathbb{N}_0$ by
\begin{equation}
\begin{split}
\label{eq:Volt}
h_{m}(t_0,\dots,t_{m})&:=O_{m}(t_0,\dots,t_m)B\text{ and }
h_{m,\operatorname{in}}(t_0,\dots,t_{m}):=O_{m}(t_0,\dots,t_m)B_{\operatorname{in}}.
\end{split}
\end{equation}
\end{defi}
From \eqref{tracenorm}, we have estimates on the difference of the trace distance of two Hankel operators with $F_i,G_i$ as in \eqref{eq:FandG},
\begin{equation}
\begin{split}
\label{eq:estmate}
&\sum_{i=0}^{\infty} \left\lVert \Delta(W_iF_i) \right\rVert_{\text{TC}} \le 2\left\lVert \Delta(H^{\operatorname{bil}}) \right\rVert_{\text{TC}}, \ \sum_{i=0}^{\infty} \left\lVert \Delta(W_iF_{i+1}) \right\rVert_{\text{TC}} \le 2\left\lVert \Delta(H^{\operatorname{bil}}) \right\rVert_{\text{TC}} \text{ and } \\
&\sum_{i=0}^{\infty} \left\lVert \Delta(W_iG_i) \right\rVert_{\text{TC}} \le 2\left\lVert \Delta(H^{\operatorname{bil}}) \right\rVert_{\text{TC}}, \ \sum_{i=0}^{\infty} \left\lVert \Delta(W_iG_{i+1}) \right\rVert_{\text{TC}} \le 2\left\lVert \Delta(H^{\operatorname{bil}}) \right\rVert_{\text{TC}}. \\
\end{split}
\end{equation}
The \emph{bilinear Gramians} satisfy the following Lyapunov equations which hold as operator equations, i.e. without testing against elements $x,y \in D(A)$, as soon as $A$ is a bounded operator, see \eqref{eq:Lyapunoveq},
\begin{lemm}[Lyapunov equation]
For all $x,y \in D(A)$ the bilinear observability Gramian satisfies the Lyapunov equation
\begin{equation}
\begin{split}
\label{eq:Lyap0}
&\left\langle \mathscr{O}^{\operatorname{bil}}Ax,y \right\rangle_K + \left\langle \mathscr{O}^{\operatorname{bil}}x,Ay \right\rangle_K + \left\langle  \mathscr{O}^{\operatorname{bil}} Nx,Ny \right\rangle_K+  \left\langle Cx,Cy\right\rangle_K=0.
\end{split}
\end{equation}
For all $x,y \in D(A^*)$ the bilinear reachability Gramian satisfies the Lyapunov equation
\begin{equation}
\begin{split}
\label{Lyap1}
&\left\langle \left(\mathscr{P}^{\operatorname{bil}}-B_{\operatorname{in}}B_{\operatorname{in}}^*\right)A^*x,y \right\rangle_K + \left\langle  \left(\mathscr{P}^{\operatorname{bil}}-B_{\operatorname{in}}B_{\operatorname{in}}^*\right)x,A^*y \right\rangle_K \\
 &+ \left\langle  \left(\mathscr{P}^{\operatorname{bil}}-B_{\operatorname{in}}B_{\operatorname{in}}^*\right) N^*x,N^*y \right\rangle_K
 +  \left\langle (BB^*+B_{\operatorname{in}}B_{\operatorname{in}}^*)x,y\right\rangle_K=0.
\end{split}
\end{equation}
\end{lemm}
\begin{proof}
Both Lyapunov equations can be obtained immediately from results on bilinear theory \cite[Lemma $2.2$]{BH18} as follows: We observe that $\mathscr{P}^{\operatorname{bil}}$ can be written as a sum of two standard bilinear reachability Gramians $\mathscr{P}^{\operatorname{bil}}_{a\vert b}$ with control operators $B_a=B$ and $B_b=B_{\operatorname{in}}$, respectively, and the projection $B_{\operatorname{in}}B_{\operatorname{in}}^*$ added to them.
Each of the Gramians $\mathscr{P}^{\operatorname{bil}}_{a\vert b}$ satisfies a Lyapunov equation \cite[Lemma $2.2$]{BH18} with $i\in \left\{a,b\right\}$
\begin{equation}
\begin{split}
&\left\langle \mathscr{P}_{i}^{\operatorname{bil}} A^*x,y \right\rangle_K + \left\langle \mathscr{P}^{\operatorname{bil}}_i x,A^*y \right\rangle_K + \left\langle  \mathscr{P}^{\operatorname{bil}}_i N^*x,N^*y \right\rangle_K
 +  \left\langle B_{i} B_i^*x,y\right\rangle_K=0.
\end{split}
\end{equation}
Adding them together and using that $\mathscr{P}^{\operatorname{bil}} - B_{\operatorname{in}}B_{\operatorname{in}}^* = \mathscr{P}_a^{\operatorname{bil}} + \mathscr{P}_b^{\operatorname{bil}}$ yields the claim. The Lyapunov equation for the observability Gramian coincides with the one in \cite{BH18}.
\end{proof}
The Lyapunov equations \eqref{eq:Lyap0} and \eqref{Lyap1} imply the following two interpretations for Gramians with different kind of delays: 
\begin{prop}
\label{homsys}
All elements $(\varphi_0,f) \in \operatorname{ker}(\mathscr{O}) \times L^2((-r,0),\operatorname{ker}(\mathscr O))$ are unobservable under the evolution of the homogeneous system. That is, the solution to 
\begin{equation}
\begin{split}
\label{eq:homogeneousevoleq}
\varphi'(t)=A\varphi(t)+N(K\varphi)(t),\ \text{ for }t>0
\end{split}
 \end{equation} 
satisfies $C\varphi(t)=0$ for all $t>0$ for $\tau>0$ and $g:[-r,0]\rightarrow \mathbb R$ as in \eqref{eq:typdelay}, where
\begin{equation}
\label{eq:K}
 (K\varphi)(t) = \varphi(t-\tau) \text{ or } (K\varphi)(t)= \int_{-r}^0g(s) \varphi(t+s) \ ds. 
 \end{equation}
\end{prop}
\begin{proof}
We start by assuming that $(\varphi_0,f) \in (\operatorname{ker}(\mathscr{O})\cap D(A)) \times H^1((-r,0),\operatorname{ker}(\mathscr O))$ with $\varphi_0 =f(0)$ first. This set is dense, as it is precisely $D(\mathcal A) \cap \left(\operatorname{ker}(\mathscr O) \times L^2((-r,0),\operatorname{ker}(\mathscr O))\right)$, cf. \eqref{eq:generator}.
As in the proof of \cite[Lemma $2.3$]{BH18} one shows that the first Lyapunov equation \eqref{eq:Lyap0} yields
\begin{equation*}
\begin{split}
&N \left(\operatorname{ker}(\mathscr{O}) \right) \subset \operatorname{ker}(\mathscr{O}), \ \operatorname{ker}(\mathscr{O}) \subset \operatorname{ker}(C), \text{ and }  A\left(\operatorname{ker}(\mathscr{O})\cap D(A) \right) \subset \operatorname{ker}(\mathscr{O}).
\end{split}
\end{equation*}

Hence, the homogeneous equation \eqref{eq:homogeneousevoleq} for $\alpha \in \operatorname{ker}(\mathscr{O})^{\perp}$ and $\varphi(t) \in \operatorname{ker}(\mathscr{O}) \cap D(A)$ satisfies
\begin{equation*}
\langle \alpha, \varphi'(t) \rangle = \langle \alpha, A \varphi(t) \rangle +   \langle \alpha, N K\varphi(t) \rangle=0.
\end{equation*}
Thus, the flow of the homogeneous problem \eqref{eq:homogeneousevoleq} leaves $\operatorname{ker}(\mathscr{O})$ invariant.
From the inclusion $\operatorname{ker}(\mathscr{O}) \subset \operatorname{ker}(C)$, we obtain
$C \varphi(t)=0.$
The statement follows then for arbitrary $\varphi:=(\varphi_0,f) \in \operatorname{ker}(\mathscr{O}) \times L^2((-r,0),\operatorname{ker}(\mathscr O))$ by approximating $\varphi$ with $D(\mathcal A_0) \cap \operatorname{ker}(\mathscr O) \times L^2((-r,0),\operatorname{ker}(\mathscr O))$ such that  $\langle \alpha,\pi_1(\mathcal{T}(t)\varphi)  \rangle= \lim_{i \rightarrow \infty} \langle  \alpha,\pi_1(\mathcal{T}(t)\varphi_i)  \rangle=0$.
\end{proof}

\begin{prop}
\label{homsys2}
The closure of the range of the reachability Gramian $\mathscr{P}$ is an invariant space of the flow of $\varphi'(t) = A\varphi(t)+  NK\varphi(t) + Bu(t)$
with $K$ as in \eqref{eq:K}.
That is, for any $(\varphi_0,f) \in \overline{\operatorname{ran}(\mathscr{P})} \times L^2((-r,0),\overline{\operatorname{ran}(\mathscr P)})$ the solution $\varphi$ stays in $\overline{\operatorname{ran}(\mathscr{P})}.$
\end{prop}
\begin{proof}
We start by assuming that $(\varphi_0,f) \in D(\mathcal A) \cap (\operatorname{ker}(\mathscr{P})^{\perp} \times H^1((-r,0),\operatorname{ker}(\mathscr{P})^{\perp}),$ first. 
From the second Lyapunov equation \eqref{Lyap1} we deduce, similarly to the previous Lemma, that
\begin{equation*} \begin{split}
&N \left(\operatorname{ker}(\mathscr{P})^{\perp}\right) \subset \operatorname{ker}(\mathscr{P})^{\perp}, \  \psi_i \in  \operatorname{ker}(\mathscr{P})^{\perp}, \text{ and }  A( \operatorname{ker}(\mathscr{P})^{\perp} \cap D(A) ) \subset  \operatorname{ker}(\mathscr{P})^{\perp}
\end{split} \end{equation*}
where we recall that $Bu =\sum_{i=1}^n \psi_i u_i.$

This shows that the homogeneous equation for $\alpha \in \operatorname{ker}(\mathscr{P})$ and $\varphi(t) \in \operatorname{ker}(\mathscr{P})^{\perp} \cap D(A)$ satisfies
\begin{equation*}
\langle \alpha, \varphi'(t) \rangle = \langle \alpha, A \varphi(t) \rangle +  \langle \alpha, N(K \varphi)(t) \rangle=0.
\end{equation*}
Thus, the flow of the homogeneous problem \eqref{eq:homogeneousevoleq} leaves $\operatorname{ker}(\mathscr{P})^{\perp}$ invariant.
From $\operatorname{ker}(\mathscr{P})^{\perp} \cap D(A) \subset \operatorname{ker}(C)$, we obtain
$C \varphi(t)=0.$
The statement follows then for general $x:=(\varphi_0,f)\in \operatorname{ker}(\mathscr{P})^{\perp} \times L^2((-r,0),\operatorname{ker}(\mathscr{P})^{\perp})$ by approximating $x$ with 
\[x_i \in D(\mathcal A) \cap \operatorname{ker}(\mathscr{P})^{\perp}\times L^2((-r,0),\operatorname{ker}(\mathscr{P})^{\perp})\] such that  $\langle \alpha,\pi_1(\mathcal{T}(t)x)  \rangle= \lim_{i \rightarrow \infty} \langle  \alpha,\pi_1(\mathcal{T}(t)x_i)  \rangle=0.$ 
The inhomogeneous equation satisfies then $\langle \pi_1(Z(t)),\alpha \rangle =\langle \pi_1(\mathcal T(t)x),\alpha \rangle+ \sum_{i=1}^n\int_0^t \langle \pi_1(\mathcal T(t-s)(\psi_i,0)),\alpha \rangle u_i(s) \ ds=0.$
\end{proof}

\section{Proof of Theorem \ref{theo:err}}
\label{sec:Proof1}
In this section we provide the proof of Theorem \ref{theo:err} which provides error estimates on the difference of two delay systems with control \eqref{eq:delay2}:
\begin{equation}
\begin{split}
\label{eq:delay3}
\varphi^{\operatorname{bild'}}(t) &= A \varphi^{\operatorname{bild}}(t) + (\Phi \varphi^{\operatorname{bild}})(t) v(t)+ Bu(t), \ t>0 \\
\varphi^{\operatorname{bild}}(0)&= \varphi_0, \quad \varphi^{\operatorname{bild}}(\sigma)= 0, \text{ for } \sigma \in (-r,0).
\end{split}
\end{equation}

As in \eqref{eq:S}, we let $K_1(t):=C\varphi_1^{\operatorname{bild}}(t)$ be defined in terms of the solution to
\begin{equation}
\begin{split}
\label{eq:delay4}
\varphi_1^{\operatorname{bild'}}(t) &= A \varphi_1^{\operatorname{bild}}(t) + (\Phi \varphi_1^{\operatorname{bild}})(t) v(t)+ Bu(t), \ t>0 \\
\varphi_1^{\operatorname{bild}}(0)&=0, \quad \varphi_1^{\operatorname{bild}}(\sigma)= 0, \text{ for } \sigma \in (-r,0)
\end{split}
\end{equation}
and $K_2(t):=C\varphi_2^{\operatorname{bild}}(t)$ where $\varphi_2^{\operatorname{bild}}$ solves
\begin{equation}
\begin{split}
\label{eq:delay5}
\varphi_2^{\operatorname{bild'}}(t) &= A \varphi_2^{\operatorname{bild}}(t) + (\Phi \varphi_2^{\operatorname{bild}})(t) v(t), \ t>0 \\
\varphi_2^{\operatorname{bild}}(0)&= \varphi_0, \quad \varphi_2^{\operatorname{del}}(\sigma)= 0, \text{ for } \sigma \in (-r,0).
\end{split}
\end{equation}

\begin{proof}[Proof of Theorem \ref{theo:err}]
In the following we write $v(t):=u(t-\tau).$ By applying the triangle inequality to the Volterra series \cite[Lemma $A.1$]{BH18}, we have using delay Volterra kernels \eqref{eq:Voltker}
\begin{equation}
\begin{split}
\label{eq:firstesm}
&\left\lVert\Delta(K_1) \right\rVert_{L^2(0,T)} \le  \left\lVert \left\lVert \Delta(\indic_{[0,\infty)}h^{\operatorname{delay}}_{0} ) \right\rVert * \indic_{[0,T)} \left\lVert u \right\rVert \right\rVert_{L^2(0,\infty)} + \\
&\sum_{k=2}^{\infty}\Bigg(\int_0^{T} \Bigg(\int_{\Delta_k(t)} \Vert \Delta ( h^{\operatorname{delay}}_{k-1}(t-s_1,\dots,s_{k-1}-s_k) v(s_1)\cdots v(s_{k-1})\cdot u(s_k) \Vert  \ ds \Bigg)^2 \ dt \Bigg)^{1/2}. 
\end{split}
\end{equation}
Our aim is to estimate the expression containing delay Volterra kernels in terms of bilinear Volterra kernels \eqref{eq:Volt}.

The first term on the right-hand side of \eqref{eq:firstesm} satisfies by Young's inequality
\begin{equation}
\begin{split}
\label{eq:firstone}
\left\lVert \left\lVert \Delta(\indic_{[0,\infty)}h^{\operatorname{delay}}_{0})  \right\rVert * \indic_{[0,T)} \left\lVert u \right\rVert \right\rVert_{L^2(0,\infty)}  
&\le  \left\lVert \Delta(h^{\operatorname{delay}}_{0}) \right\rVert_{L^1(0,\infty)}  \left\lVert u \right\rVert_{L^2((0,T))}\\
&\le  \left\lVert \Delta(h_{0}) \right\rVert_{L^1(0,\infty)}  \left\lVert u \right\rVert_{L^2(0,T)}.
\end{split}
\end{equation}
To estimate the second term on the right-hand side, we observe that by Minkowski's inequality and H\"older's inequality for $k \ge 2$
\begin{equation}
\begin{split}
\label{eq:one}
&\Bigg(\int_0^{T} \Bigg(\int_{\Delta_k(t)} \Vert \Delta (h^{\operatorname{delay}}_{k-1}(t-s_1,\dots,s_{k-1}-s_k) ) v(s_1)\cdots u(s_k) \Vert  ds \Bigg)^2 \ dt \Bigg)^{1/2}\\
&\le \int_0^{T}  \Bigg(\int_{s_1}^{T}  \Bigg(\int_{\Delta_{k-1}(s_1)}  \Vert  \Delta (h^{\operatorname{delay}}_{k-1}(t-s_1,\dots,s_{k-1}-s_k)) v(s_2)\cdots u(s_k) \Vert \ ds \Bigg)^2 dt \Bigg)^{1/2}  \vert v(s_1) \vert \ ds_1  \\
&\le \sup_{0\le s_1\le T} \Bigg(\int_{0}^{T-s_1}  \Bigg(\int_{\Delta_{k-1}(s_1)} \Vert  \Delta (h^{\operatorname{delay}}_{k-1}(t,\dots,s_{k-1}-s_k) ) v(s_2)\cdots u(s_k) \Vert \ ds \Bigg)^2  \ dt \Bigg)^{1/2} \Vert v \Vert_{L^1(0,T)}  \\
&\le \Bigg(\int_{0}^{T} \sup_{0\le s_1\le T} \Bigg(\int_{\Delta_{k-1}(s_1)} \Vert  \Delta(h^{\operatorname{delay}}_{k-1}(t,s_1-s_2, \dots,s_{k-1}-s_k) ) v(s_2)\cdots u(s_k) \Vert \ ds \Bigg)^2  \ dt \Bigg)^{1/2} \Vert v \Vert_{L^1(0,T)}.
\end{split}
\end{equation}
Thus, by applying H\"older's inequality, to the inner integral we find that using $\Vert v \Vert_{L^2} \le 1$
\begin{equation*}
\begin{split}
&\int_{\Delta_{k-1}(s_1)} \Vert  \Delta(h^{\operatorname{delay}}_{k-1}(t,s_1-s_2,\dots,s_{k-1}-s_k) ) \ \vert v(s_2)\cdots u(s_k) \vert \Vert \ ds\\
&=\int_{0}^{s_1}\dots\int_0^{s_{k-1}}\left\lVert \Delta h^{\operatorname{delay}}_{k-1}(t,s_1-s_2,\dots,s_{k-1}-s_{k}) \ v(s_2) \cdots u(s_{p}) \right\rVert_{\mathcal L(\mathbb R^n, \mathbb R^m)} \ ds_{k}\cdots ds_2 \\
&\le  \left(\int_{(0,T)^{i-2}}\left(\int_{0}^{\infty}\left(\int_{(0,T)^{k-i}} \left\lVert \Delta h_{k-1}^{\operatorname{delay}}(t,s,r,q) \right\rVert^2_{\mathcal L(\mathbb R^n, \mathbb R^m)}  \ dq \right)^{\frac{1}{2}} \ dr \right)^2 \ ds  \right)^{\frac{1}{2}} \Vert u \Vert_{L^{\infty}(0,T)}
\end{split} 
\end{equation*}
Applying Minkowski's integral inequality to this expression, leads after a change of variables to remove the delay, together with \eqref{eq:one} to
\begin{equation*}
\begin{split}
&\Bigg(\int_0^{T} \Bigg(\int_{\Delta_k(t)} \Vert \Delta (h_{k-1}^{\operatorname{delay}}(t-s_1,\dots,s_{k-1}-s_k)B ) v(s_1)\cdots u(s_k) \Vert  ds \Bigg)^2 \ dt \Bigg)^{1/2}\\
&\le   \int_0^{\infty}  \left( \int_{(0,T)^{k-1}} \left\lVert \Delta h_{k-1}(q_1,\dots,q_{i-1},r,q_i,\dots,q_{k-1}) \right\rVert^2_{\mathcal L(\mathbb R^n, \mathbb R^m)}  dq  \right)^{\frac{1}{2}} \ dr  \Vert u \Vert_{L^{\infty}(0,T)}\Vert v \Vert_{L^{1}(0,T)}.
\end{split} 
\end{equation*}
Thus, we have together with \eqref{eq:firstone}
\begin{equation*}
\begin{split}
\Vert \Delta(K_1) \Vert_{L^2(0,T)} 
&\le \sum_{i=1}^{\infty} \left(\left\lVert \Delta(h _{2i-1}) \right\rVert_{L^{1}_iL^{2}_{2i-1}(\operatorname{HS})}+ \left\lVert \Delta( h _{2i-2}) \right\rVert_{L^{1}_iL^{2}_{2i}(\operatorname{HS})} \right)\cdot\\
&\qquad \operatorname{max} \Bigg\{\Vert u \Vert_{L^{2}(0,T)}\Vert u \Vert_{L^{\infty}(0,T)}\Vert u \Vert_{L^{1}(0,T)}\Bigg\} \\
&\le  4 \left\lVert \Delta(H^{\operatorname{bil}}) \right\rVert_{\text{TC}}\operatorname{max} \left\{\Vert u \Vert_{L^{2}(0,T)}, \Vert u \Vert_{L^{\infty}(0,T)}\Vert v \Vert_{L^{1}(0,T)}\right\}. 
 \end{split}
\end{equation*}

The difference $\Delta(K_2)$ can be bounded, in terms of $\textbf{w}$ as introduced in Theorem \ref{theo:err}, using the Cauchy-Schwarz inequality, and Minkowski's integral inequality as
 \begin{equation*}
\begin{split}
\left\lVert \Delta K_2 \right\rVert_{L^2(0,T)}&\le  \left\lVert \Delta(CT B_{\operatorname{in}})\textbf{w} \right\rVert_{L^2(0,T)} \\
&\quad + \sum_{i=1}^{\infty} \left\lVert \int_{\Delta_i(\cdot)}  \left\lVert  \Delta \left(h^{\operatorname{delay}}_{i,\operatorname{in}}(\cdot-s_1,\cdots,s_{i-1}-s_i,s_i)\right)(v)  \right\rVert v(s_1)\cdots v(s_i) \ ds \right\rVert_{L^2(0,T)} \\
&\le \left\lVert \Delta(CT B_{\operatorname{in}}) \right\rVert_{L^2(0,T)}\left\lVert \varphi_0 \right\rVert_K \\
&\quad +  \sum_{i=1}^{\infty} \int_0^T \left\lVert \Delta \left(h_{i,\operatorname{in}}(\cdots,s,\cdots)\right)   \right\rVert_{L^2((0,\infty)^{i},\mathcal H)} \ ds \ \left\lVert v \right\rVert_{L^{\infty}(0,T)}\left\lVert \varphi_0 \right\rVert_K.
\end{split}
\end{equation*}
The statement then follows from \eqref{eq:estmate} by inserting the estimate
\[ \left\lVert \Delta(CT B_{\operatorname{in}}) \right\rVert_{L^2(0,T)} \le \left\lVert \Delta(W_0G_0) \right\rVert_{\operatorname{HS}}\le \left\lVert \Delta(W_0G_0) \right\rVert_{\text{TC}} \] 
into the estimate on the difference of all Volterra kernels \eqref{eq:Volt}, cf. \cite[Lemma $4.2$]{BH18},
\begin{equation*}
\begin{split}
\left\lVert \Delta K_2 \right\rVert_{L^2(0,T)}& \le \Bigg(\left\lVert \Delta( h _{2i,\operatorname{in}}) \right\rVert_{L^{2}}+  \sum_{i=1}^{\infty} \left(\left\lVert \Delta(h _{2i-1,\operatorname{in}}) \right\rVert_{L^{1}_iL^{2}_{2i-1}(\operatorname{HS})}+ \left\lVert \Delta( h _{2i,\operatorname{in}}) \right\rVert_{L^{1}_iL^{2}_{2i-2}(\operatorname{HS})} \right) \Bigg) \\
&  \cdot \operatorname{max} \left\{1,\left\lVert v \right\rVert_{L^{\infty}(0,T)} \right\}\left\lVert \varphi_0 \right\rVert_K \le  4 \left\lVert \Delta(H^{\operatorname{bil}}) \right\rVert_{\text{TC}}  \operatorname{max} \left\{1,\left\lVert v \right\rVert_{L^{\infty}(0,T)} \right\} \left\lVert \varphi_0 \right\rVert_K. 
\end{split}
\end{equation*}
\end{proof}

\bigskip

We now explain how to extend the previous error bound to systems \eqref{eq:delay1} without a control function $v:$

\begin{rem}
By redefining 
\begin{equation}
    \begin{split}
    \label{eq:redef}
    \hat \Phi&= \sqrt{T_0} \Phi, \ \hat N = \sqrt{T_0}N,\text{ and }v(t)=T_0^{-1/2},
            \end{split}
\end{equation} 
\eqref{eq:delay3} on the time-interval $[0,T_0]$ becomes the uncontrolled delay system
\begin{equation}
\begin{split}
\label{eq:delay6}
\varphi^{\operatorname{bild'}}(t) &= A \varphi^{\operatorname{bild}}(t) + (\hat \Phi \varphi^{\operatorname{bild}})(t) + Bu(t), \ t\in (0,T_0] \\
\varphi^{\operatorname{bild}}(0)&= \varphi_0, \quad \varphi^{\operatorname{bild}}(\sigma)= 0, \text{ for } \sigma \in (-r,0)
\end{split}
\end{equation}
with $\Vert v \Vert_{L^2(0,T_0)}=1.$
\end{rem}

Hence, the following corollary follows straight from Theorem \ref{theo:err}

\begin{corr}
\label{corr:err}
Let $\mathcal{H}\simeq \mathbb{R}^m$ and consider the difference of two solutions to \eqref{eq:delay6} on a time interval $[0,T_0]$. If we then interpret these two equations as solutions to \eqref{eq:delay3} with redefined \eqref{eq:redef} and  $M\left\lVert \hat N \right\rVert /\sqrt{2\omega}<1$, such that the Volterra series converges \cite[Lemma $A.1$]{BH18}, then, for control functions $u \in L^2((0,T_0),\mathbb R^n)$, initial states $\varphi_0 = \sum_{i=1}^k \langle \textbf{w} , \widehat{e_i} \rangle \phi_i$ and $\widetilde{\varphi}_0:= \sum_{i=1}^k \langle  \textbf{w} , \widehat{e_i} \rangle \widetilde{\phi}_i$, and zero history function, it follows that 
\begin{equation}
\begin{split}
\label{eq:erresm3}
\left\lVert \Delta(C\varphi^{\operatorname{bild}}) \right\rVert_{L^2((0,\infty), \mathbb R^m)} &\le 4   \left\lVert \Delta(H^{\operatorname{bil}}) \right\rVert_{\operatorname{TC}} \Bigg(\left\lVert \varphi_0 \right\rVert_{X} \operatorname{max} \left\{ 1, T_0^{-1/2} \right\} \\
&\qquad +\operatorname{max} \left\{\Vert u \Vert_{L^{2}(0,T_0)}, \sqrt{T_0} \right\}\left\lVert u \right\rVert_{L^{\infty}(0,T_0)} \Bigg).
\end{split}
\end{equation}
\end{corr}
 
\section{Stochastic delay differential equations}
\label{sec:SDDE}
The balanced truncation theories of bilinear and stochastic systems with multiplicative noise have many features in common \cite{BH18}. In particular, the Gramians for both systems obey the same Lyapunov equations. We now demonstrate that the same is true for stochastic delay equations (SDDEs): Consider a probability space $(\Omega, \mathcal F, \mathbb P)$ with filtration $\mathcal F_t$ induced by a one-dimensional Brownian motion $(W(t))_t.$ We then study the stochastic delay differential equation (SDDE) for a process $X^{\operatorname{sdde}}_{t}$ taking values in $\mathbb R^d$ with matrices $A,N \in \mathbb R^{d \times d},\ B \in \mathbb R^{d \times n}$ and a control $u \in L^2_{\text{ad}}(\Omega \times (0,\infty); \mathbb R^n)$ which, cf. \eqref{eq:sdde}, is given by
\begin{equation}
\begin{split}
\label{eq:delay21}
dX^{\operatorname{sdde}}_t &= \left(A X^{\operatorname{sdde}}_{t}+ Bu(t)\right) \ dt +N X^{\operatorname{sdde}}_{t-\tau} \ dW_t, \ t>0 \\ 
X^{\operatorname{sdde}}_0&=\xi, \ X^{\operatorname{sdde}}_t = f_t \text{ for }-r\le t < 0.
\end{split}
\end{equation}
where $f_t$ is a $\mathcal F_0$-measurable $C([-r,0],\mathbb R)$-valued random variable for which the second moment exists uniformly on $[-r,0]$ and $\xi \in L^2(\Omega,\mathcal F_0)$. The solution to \eqref{eq:delay21} satisfies then
\begin{equation}
\label{eq:solution}
X^{\operatorname{sdde}}_t = T(t)\xi +  \int_0^t T(t-s)  NX^{\operatorname{sdde}}_{s-\tau} \ dW_s+ \int_0^t T(t-s)Bu(s) \ ds.
\end{equation}
We are interested in an output variable $CX_t$ where the observation operator $C$ is a matrix of appropriate size.
This identity shows that $X_t$ is a semi-martingale, since $X_t$ is also $(\mathcal F_t)$ adapted, has continuous paths for $t \ge 0,$ and satisfies for all $T>0$ the uniform square-integrability condition $\mathbb E  \left(\sup_{t \in (-r,T)} \left\lVert X_t \right\rVert^2 \right) < \infty.$ Moreover, $X_t$ is a $C^b$-Feller process \cite{BS17}.
In particular, we can define for any $C([-r,0])$-valued process $\Psi$ and initial condition $\xi \in L^2(\Omega,\mathcal F)$ as in \eqref{eq:sdde} the flow $\Phi^{\operatorname{sdde}}_{\Psi}(t)(\xi):=X_t^{\text{hom}}$ where $X_t^{\text{hom}}$ is the homogeneous part of \eqref{eq:sdde}, i.e.\@ the process $X^{\operatorname{sdde}}$ with $u \equiv 0.$ 
We record that once $\Psi \equiv 0$ the flow becomes linear in the initial state $\xi.$ We also write $\Phi^{\operatorname{sdde}}_{\Psi}(t,s)(\xi)$ to denote the process started at time $s$ from $\xi.$
It follows then directly from \eqref{eq:solution} that for $\Psi=0$, the solution to \eqref{eq:sdde} is given by the variation of constant formula 
\begin{equation}
\label{eq:VOCF}
X_t^{\operatorname{sdde}} =  \int_0^t \Phi^{\operatorname{sdde}}_0(t,r)   Bu(r) dr+\Phi^{\operatorname{sdde}}_0(t)(\xi).
\end{equation}
\begin{ass}
We assume the flow $\Phi^{\operatorname{sdde}}_{0}$ to be exponentially stable in mean square sense, i.e. there are $C,\kappa>0$ such that for all $x \in \mathbb R^m$ we have $ \mathbb E \left\lVert \Phi^{\operatorname{sdde}}_{0}(t,0)x \right\rVert^2 \le C e^{-\kappa t} \left\lVert x \right\rVert^2.$ 
\end{ass}
\begin{disc}[Exponential stability]
The exponential stability of SDEs and SDDEs has been thoroughly addressed and relevant results for our framework can for example be found in \cite{MS97}. In particular, \cite[Ex. $4.1$]{MS97} implies that the SDDE
\begin{equation}
dX_t = A X_t \ dt + \sum_{i=1}^l \left(A_i X_t + N_i X_{t-\tau} \right) dW_t^i 
\end{equation}
possesses an exponentially stable flow if there are two positive definite matrices $Q,G$ such that the Lyapunov equation holds
\begin{equation}
\label{eq:lapstab}
GA + A^T G + \sum_{i=1}^l \left( A_i + N_i \right)^T G \left(A_i+N_i \right) = -Q  
\end{equation}
and the delay satisfies the smallness condition $\tau < \frac{\sqrt{\delta_1^2+ \delta_3 \left\lVert A \right\rVert^2}-\delta_1}{2\left\lVert A \right\rVert^2 }$
where in terms of the smallest eigenvalue $\lambda_{\operatorname{min}}(T)$ of a matrix $T$ satisfies
\begin{equation}
\begin{split}
\delta_1&:=\sum_{i=1}^l \left(\left\lVert A_i \right\rVert^2 + \left\lVert N_i \right\rVert^2 \right), \quad 
\delta_2:=2 \left\lVert G \right\rVert  \sqrt{2\delta_1 \sum_{i=1}^l \left\lVert N_i \right\rVert^2} , \text{ and }\\
\delta_3&:=\left(\frac{\sqrt{\delta_2^2+4 \lambda_{\operatorname{min}}(Q) \left\lVert G \right\rVert \sum_{i=1}^l \left\lVert N_i \right\rVert^2 }-\delta_2}{2\left\lVert G \right\rVert\sum_{i=1}^l\left\lVert N_i \right\rVert^2}\right)^2.
\end{split}
\end{equation}
\end{disc}
\subsection{Gramians}
Let us now introduce the Gramians and the Hankel operator for the SDDE \eqref{eq:sdde}:
\begin{defi}
\label{def:gram}
The \emph{observability map} $W^{\operatorname{sdde}} \in \mathcal L( K, L^2(\Omega_{(0,\infty)}; \mathcal H))$ and \emph{reachability map} $R^{\operatorname{sdde}} \in \mathcal L(L^2(\Omega_{(0,\infty)}; \mathbb R^n) \oplus \mathbb R^k,K)$ are defined as 
\[(W^{\operatorname{sdde}}x)(t)= C\Phi_{0}^{\operatorname{sdde}}(t)x \text{ and } R^{\operatorname{sdde}}(f,v)= \mathbb E\int_0^{\infty} \Phi^{\operatorname{sdde}}_{0}(s)Bf(s) \ ds +B_{\operatorname{in}}v . \]
The \emph{Hankel operator} is defined as $H^{\operatorname{sdde}}=W^{\operatorname{sdde}}R^{\operatorname{sdde}}.$
The \emph{stochastic reachability Gramian} is defined as $\mathscr P^{\operatorname{sdde}}= R^{\operatorname{sdde}}R^{\operatorname{sdde}*}$ where
\[\mathscr P^{\operatorname{sdde}}=\mathbb E \int_0^{\infty} (\Phi^{\operatorname{sdde}}_{0}(s)B)(\Phi^{\operatorname{sdde}}_{0}(s)B)^T \ ds + B_{\operatorname{in}}B_{\operatorname{in}}^*.\]
The \emph{observability Gramian} is defined as $\mathscr O^{\operatorname{sdde}}= W^{\operatorname{sdde}*}W^{\operatorname{sdde}}$ where 
\[\mathscr O^{\operatorname{sdde}} = \mathbb E \int_0^{\infty} \Phi^{\operatorname{sdde}}_{0}(s)^T C^T C \Phi^{\operatorname{sdde}}_{0}(s) \ ds.\]
\end{defi}

\begin{prop}
The observability and reachability Gramians satisfy the following Lyapunov equations
\begin{equation}
\begin{split}
\label{eq:firLyap}
&BB^T +(\mathscr P^{\operatorname{sdde}}-B_{\operatorname{in}}B_{\operatorname{in}}^T) A^T+ A  (\mathscr P^{\operatorname{sdde}}-B_{\operatorname{in}}B_{\operatorname{in}}^T)+  N(\mathscr P^{\operatorname{sdde}}-B_{\operatorname{in}}B_{\operatorname{in}}^T)  N^T=0 \\
& \text{ and }\quad C^TC +N^T\mathscr O^{\operatorname{sdde}} \ N+\mathscr O^{\operatorname{sdde}} A+ A^T \mathscr O^{\operatorname{sdde}}=0.
\end{split}
\end{equation}
\end{prop}
\begin{proof}
Stochastic integration by parts of the process $X_t = \Phi^{\operatorname{sdde}}_0(t)\xi$ brings
\begin{equation}
\begin{split}
\label{eq:SIBP}
\mathbb E\left(X_t X_t^T \right) &= \mathbb E(\xi\xi^T)+\int_0^t \mathbb E\left(X_s X_s^T \right) \ ds \ A^T+ A \int_0^t \mathbb E\left(X_s X_s^T \right) \ ds\\
& \quad +N \int_0^t \mathbb E\left(X_{s-\tau} X_{s-\tau}^T \right) \ ds \ N^T.
\end{split}
\end{equation}
We then take the initial state $x_0=\psi_i$ with $\psi_i$ as in $Bu =\sum_{i=1}^n \psi_i u_i$, pass to the limit $t \rightarrow \infty$, and perform a simple change of variables $s':=s-\tau$ in the last integral.
This immediately yields \eqref{eq:firLyap}.
Let us now find a Lyapunov equation for the observability Gramian. Stochastic integration by parts shows that 
\begin{equation}
\begin{split}
\label{eq:computation}
&\mathbb E\left(\Phi^{\operatorname{sdde}}_{0}(t) \mathscr O^{\operatorname{sdde}} \Phi^{\operatorname{sdde}}_{0}(t) \right) = \mathscr O^{\operatorname{sdde}} + \mathbb E \left( \int_0^t \Phi^{\operatorname{sdde}}_{0}(s)^T \mathscr O^{\operatorname{sdde}} A \Phi^{\operatorname{sdde}}_{0}(s) \ ds \right)\\
&+ \mathbb E \left( \int_0^t \Phi^{\operatorname{sdde}}_{0}(s)^T A^T \mathscr O^{\operatorname{sdde}}  \Phi^{\operatorname{sdde}}_{0}(s) \ ds \right) + \mathbb E \left( \int_0^t \Phi^{\operatorname{sdde}}_{0}(s-\tau)^T N^T\mathscr O^{\operatorname{sdde}} N\Phi^{\operatorname{sdde}}_{0}(s-\tau) \right) \ ds \\
&= \mathscr O^{\operatorname{sdde}} + \mathbb E \left( \int_0^t \Phi^{\operatorname{sdde}}_{0}(s)^T \mathscr O^{\operatorname{sdde}} A \Phi^{\operatorname{sdde}}_{0}(s) \ ds \right) + \mathbb E \left( \int_0^t \Phi^{\operatorname{sdde}}_{0}(s)^T A^T \mathscr O^{\operatorname{sdde}}  \Phi^{\operatorname{sdde}}_{0}(s) \ ds \right)  \\
& \quad + \mathbb E \left( \int_0^t \Phi^{\operatorname{sdde}}_{0}(s)^T N^T\mathscr O^{\operatorname{sdde}} N\Phi^{\operatorname{sdde}}_{0}(s) \right) \ ds.
\end{split}
\end{equation}
Thus, if $X$ satisfies the Lyapunov equation $C^TC +N^T X \ N+X A+ A^T X=0$, we find by \eqref{eq:computation} and the definition of the observability Gramian, Definition \ref{def:gram}, that indeed $\mathscr O^{\operatorname{sdde}}=X.$ 
\end{proof}

\begin{rem}
The proof of the Lyapunov equations, in particular \eqref{eq:SIBP} and \eqref{eq:computation}, show that for non-constant history function $f_t$, the Lyapunov equations generalize to
\begin{equation}
\begin{split}
&BB^T+ \int_{-\tau}^0 N\mathbb E(f_t  f_t^*)N^T \ dt +(\mathscr P^{\operatorname{sdde}}-B_{\operatorname{in}}B_{\operatorname{in}}^T) A^T+ A  (\mathscr P^{\operatorname{sdde}}-B_{\operatorname{in}}B_{\operatorname{in}}^T)\\
&+  N(\mathscr P^{\operatorname{sdde}}-B_{\operatorname{in}}B_{\operatorname{in}}^T)  N^T=0 \quad\text{ and }  \\
&C^TC +N^T\mathscr O^{\operatorname{sdde}} \ N+ \int_{-\tau}^0\mathbb E(f_t N^T \mathscr O^{\operatorname{sdde}} N f_t^*) \ dt +\mathscr O^{\operatorname{sdde}} A+ A^T \mathscr O^{\operatorname{sdde}}=0.
\end{split}
\end{equation}

\end{rem}

In addition, the observability Gramian defines the $L^2$ energy of the uncontrolled process, i.e. with $u \equiv 0:$ 
\[ \mathcal E_{\operatorname{output}}:=\int_0^{\infty} \mathbb E\left \lVert C\Phi_0^{\operatorname{sdde}}(t,0) x \right\rVert^2 \ dt = \langle \mathscr O^{\operatorname{sdde}}x,x \rangle.\]
For the reachability Gramian less direct interpretations, already studied in the context of stochastic system with multiplicative, cf. \cite{BH18}, can be stated for stochastic delay equations, as studied here, too.

\bigskip

\begin{lemm} 
\label{lemmasdde}
The difference of Hankel operators $\Delta(H^{\text{sdde}})$ satisfies for two independent Wiener processes \begin{equation}
\begin{split}
\label{eq:estmsdde}
 \left\lVert \Delta \left(C\Phi^{\operatorname{sdde}} B_{\operatorname{in}}\right)\right\rVert_{L^2(\Omega_{(0,\infty)},\operatorname{HS}(\mathbb R^k, \mathbb R^m))} &\le  \left\lVert \Delta\left(H^{\operatorname{sdde}}\right) \right\rVert_{\operatorname{HS}} \text{ and }\\
\left\lVert \Delta \left(C\Phi^{\operatorname{sdde}} B\right) \right\rVert_{L^1_tL^2_{\omega}(\Omega_{(0,\infty)},\operatorname{HS}(\mathbb R^n, \mathbb R^m))} &\le 2\left\lVert \Delta \left(H^{\operatorname{sdde}} \right) \right\rVert_{\operatorname{TC}}.
\end{split}
\end{equation}
\end{lemm}
\begin{proof}
The first bound follows immediately from the definition of the Hilbert-Schmidt norm. The second bound follows along the lines of the proof of \cite[Theorem $3$, (5.11)]{BH18}.
\end{proof}

We are now ready to give the proof of the error bound for stochastic delay differential equations:

\begin{proof}[Proof of Theorem \ref{theo:sdde}]
By Young's inequality we have for $v \in \mathbb R^k$ as in the statement of the theorem that $\left\lVert v \right\rVert_{\mathbb R^k}= \left\lVert \xi \right\rVert_{L^2(\Omega,K)}$ and 
\begin{equation*}
\begin{split}
  &\left\lVert \Delta \left(CX^{\operatorname{sdde}} \right) \right\rVert_{L^2(\Omega_{(0,T)},\mathbb R^m)} \\
  &\overset{\eqref{eq:VOCF}}{\le}  \left\lVert \Delta(C\Phi^{\operatorname{sdde}} B_{\operatorname{in}})(v)\right\rVert_{L^2(\Omega_{(0,T)},\mathbb R^m)} + \left\lVert \left\lVert \Delta \left(\indic_{[0,\infty)}C \Phi^{\operatorname{sdde}} B \right) \right\rVert * \indic_{[0,T)} \left\lVert u \right\rVert \right\rVert_{L^2(\Omega_{(0,T)},\mathbb R^m)}  \\
 &\le \left\lVert \Delta(C \Phi^{\operatorname{sdde}} B_{\operatorname{in}})\right\rVert_{L^2(\Omega_{(0,\infty)},\mathbb R^{m \times k})}  \left\lVert \xi \right\rVert_{L^{2}(\Omega;K)} \\
 & \quad +  \left\lVert \Delta \left(C \Phi^{\operatorname{sdde}} B\right) \right\rVert_{L^1_tL^2_{\omega}(\Omega_{(0,\infty)},\mathbb R^{m \times n})}  \left\lVert u \right\rVert_{L^{\infty}_{\omega}L^2_t(\Omega_{(0,T)},\mathbb R^n)} \\
 &\overset{\operatorname{Lemma} \ \ref{lemmasdde}}{\le}  \left\lVert \Delta(H^{\operatorname{sdde}}) \right\rVert_{\operatorname{TC}} \left(\left\lVert \xi \right\rVert_{L^{2}(\Omega;K)} + 2\left\lVert u \right\rVert_{L^{\infty}_{\omega}L^2_t(\Omega_{(0,T)},\mathbb R^n)} \right).
 \end{split}
 \end{equation*} 
 \end{proof}

\bigskip
\bigskip 

\section{Applications and Examples}
\label{sec:Examples}

We conclude by analyzing three applications of the model order reduction methods studied in this article -- the corresponding code can be found at {\url{https://github.com/lorenzrichter/balanced-truncation}.} 
\label{sec:examples}
\smallsection{Example: \ Linearized Stuart-Landau Oscillator} \cite{PYPT,WYH,ZZ} 
The Stuart-Landau oscillator system is a coupled nonlinear network and a model for phenomena such as chaos or synchronization in large physical or biological systems modeling a finite speed of propagation. 
The \emph{dissipative} system is described by a parameter $\Re(\alpha)<0$ and has an equilibrium solution $\varphi=0$, i.e.\@ all oscillators at rest. The dynamics of the Stuart-Landau system with unidirectional nearest-neighbor interaction is described by the following coupled system of nonlinear differential equations
\begin{equation}
\begin{split}
\label{eq:dynamics}
\varphi_j'(t) = \alpha \varphi_j(t) - \varphi_j(t) \vert \varphi_j(t) \vert^2 + \varphi_{(j+1) \bmod {N}} (t-\tau), \ j=1,\dots, N.
\end{split}
\end{equation}
with $\tau>0$. 
By linearizing the dynamics \eqref{eq:dynamics} around the equilibrium solution, we obtain the following coupled system of delay equations
\begin{equation}
\begin{split}
\varphi_j'(t) = \alpha \varphi_j(t)+ \varphi_{(j+1) \bmod {N}}(t-\tau), \ j=1,\dots, N.
\end{split}
\end{equation}
Hence, $A=\alpha \indic$ and therefore $T(t)=e^{\alpha \indic t}$ and thus $M=1$ and $\omega=\vert \Re(\alpha) \vert.$ In a numerical example, let us consider $d = 50$ and take $\alpha = -1.2, T = 2, \tau = 0.1$. We discretize the dynamics with a simple forward Euler scheme using the stepsize $\Delta t = 0.01$ and a random initial point $\varphi(0) \sim \mathcal{N}(0, \sqrt{0.5}\indic)$. We choose a history function $f=0$ and do balanced truncation as described above while varying the dimensions of the corresponding reduced systems. Figure \ref{fig: stuart-landau} displays the $L^2$ errors when reducing to $r$-dimensional systems and compares to the corresponding bounds from Corollary \ref{corr:err}. We see that the measured $L^2$ error decays rapidly and that the bound seems to be rather conservative. The two right plots in the same figure show some components of the full and the reduced systems when either choosing $r=2$ or $r=6$. In the latter case we see almost full agreement of the full and reduced trajectories.

\begin{figure}
\centering
\includegraphics[width=1.0\linewidth]{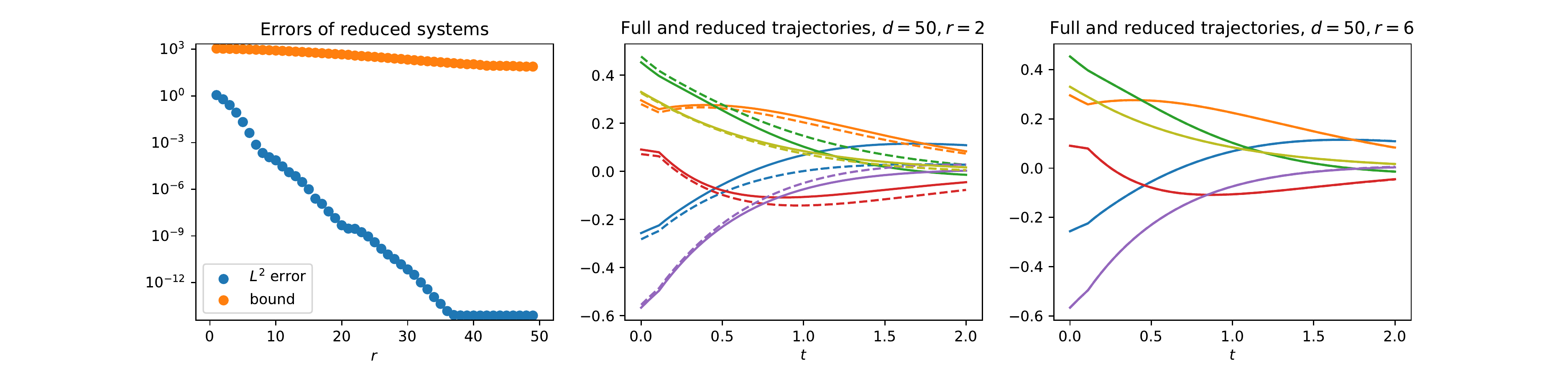}
\caption{Balanced truncation of the Stuart-Landau oscillator with delay. We display the $L^2$ errors with their corresponding bounds as well as some components of the full (solid lines) and reduced (dashed lines) trajectories.}
\label{fig: stuart-landau}
\end{figure}

\smallsection{Example: \ Generalized Langevin equation} \cite{K,LLL,M}
Next, we consider a collection of particles obeying a damped Newtonian dynamics
\begin{equation}
\label{eq:Hamdyn}
 M \ddot{x}(t) + C \dot{x}(t)  + K x(t) = 0. 
 \end{equation}
In \eqref{eq:Hamdyn} both the coupling matrix $K= L_1L_1^*$ and mass matrix $M = L_2L_2^*$ are assumed to be strictly positive matrices and the friction matrix $F$ is a positive semidefinite matrix with the property that the form $x \mapsto \langle x,F x \rangle$ is non-degenerate on every eigenspace of $M^{-1}K.$

Introducing new coordinates $y(t) :=(L_1^*x(t),L_2^*\dot{x}(t))$, we can define a generator
\[A = \left(\begin{matrix} 0 & L_1^* L_2^{-1*} \\ -L_2^{-1} L_1 & -L_2^{-1} F L_2^{-1*} \end{matrix} \right) \]
such that the semigroup defined as $T(t):=e^{tA}$, associated with \eqref{eq:Hamdyn}, is exponentially stable. Thus, the solution to \eqref{eq:Hamdyn} is given by $y(t)= T(t) y_0.$
To model particle motion in contact with a heat bath the following generalized Langevin equation (GLE) has been proposed
\begin{equation}
\label{eq:GLE}
M \ddot{x}(t) + F \dot{x}(t)  + K x(t) =- \int_{0}^t \gamma(t-s)\Gamma_0 \dot{x}(s) \ ds-B_0u(t), 
\end{equation}
where $B_0u(t)$ describes an external fluctuation force. To include a \emph{memory effect} in the dissipation the friction is perturbed by a convolution between a kernel $\gamma \Gamma_0$ and the velocity of the particles, where $\Gamma_0$ is a matrix representing the typical scale of friction. 
In order to cast \eqref{eq:GLE} in a form that resembles more the type of delay equations we have been studying in this article, we use matrices $N:= \operatorname{diag} \left(0, \Gamma_0 \right)$ and $B:= \operatorname{diag} \left(0, B_0 \right),$
such that the dynamics \eqref{eq:GLE} takes the form
\begin{equation}
\begin{split}
\label{eq:GLE1}
y'(t) &= A y(t) + \int_0^t  \gamma(t-s) Ny(s)\ ds + Bu(t).
\end{split}
 \end{equation}
We then consider fractional Brownian motion (fBm) $B_H$ with Hurst parameter $H \in (0,1)$ and correlation $\mathbb E(B_t^H B_s^H) = \frac{1}{2} \left( s^{2H}+t^{2H}-\vert t-s \vert^{2H} \right).$
Moreover, for $t>s=0$ the distributional time derivative of fBm has correlation coefficients $\mathbb E(\dot{B}_t^H \dot{B}_0^H) = H(2H-1)t^{2(H-1)}$ which tend to zero for large $t$. For Hurst parameters $H \in (1/2,1)$ the correlation coefficient is also integrable at zero. We then use a cut-off function to truncate the small tail (neglecting memory effects from time more than a distance $r$ away from the current time) and define the history kernel in \eqref{eq:GLE1} to be $\gamma(t) :=\indic_{[r,0]}(t) t^{2(H-1)}.$ Thus, 
\begin{equation}
\begin{split}
y'(t) &= A y(t) + \int_{t-r}^t  \gamma(t-s) Ny(s)\ ds + Bu(t)\\
&\approx A y(t) +  r\gamma(r/2) Ny(t-r/2) + Bu(t).
\end{split}
\end{equation}
Now, the high-dimensional system \eqref{eq:GLE1} can be reduced using balanced truncation as described in this article. For a numerical illustration, let us choose $d=100$ $M = \indic + 0.1\operatorname{diag}(a_1, \dots, a_d), F = \indic + 0.1\operatorname{diag}(a_1, \dots, a_d), K = \indic + \left(|a_{ij}/2|\right)_{i,j=1}^{d}, \Gamma_0 = \indic$, where $a_i, a_{ij} \sim \mathcal{N}(0, 1)$ are chosen i.i.d. for all $i, j \in \{1, \dots, d\}$. We further take $u(t) = \sin(20 t) \mathbf{1}, C = \indic, T = 10, \tau = 0.1$ and plot components of the full and reduced systems in Figure \ref{fig: Langevin}, again considering different dimensions $r$ for the reduced models. In spite of a dimension reduction from $d = 100$ to $r = 10$ the trajectories look almost the same in the right plot.

\begin{figure}
\centering
\includegraphics[width=1.0\linewidth]{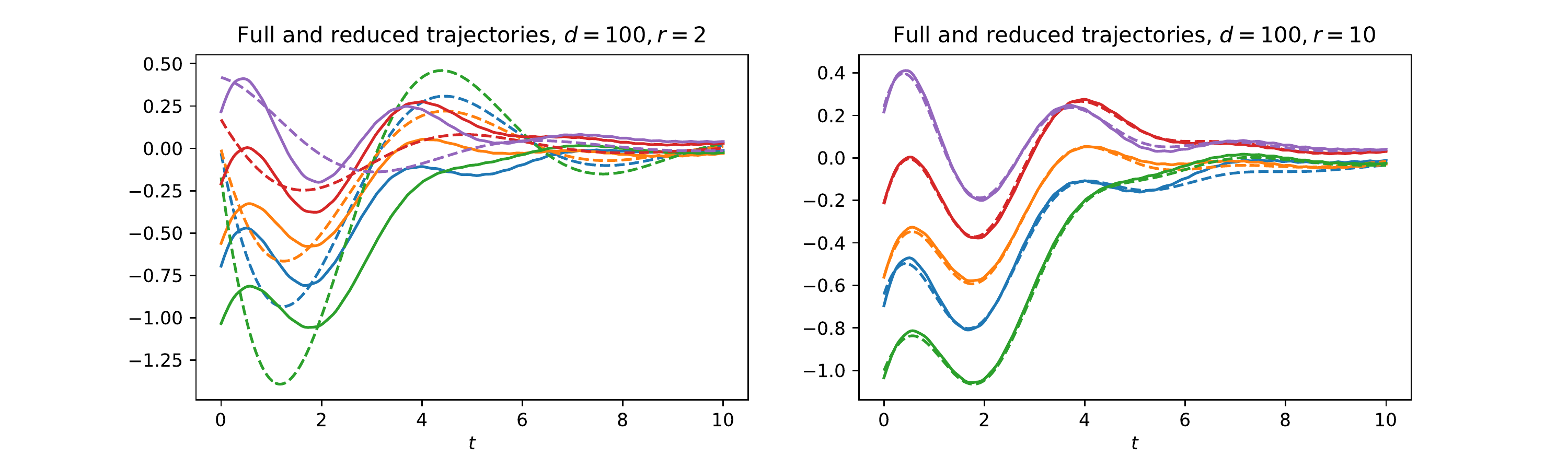}
\caption{Components of the full (solid lines) and reduced Langevin dynamics with different dimensions $r$ of the reduced systems (dashed lines).}
\label{fig: Langevin}
\end{figure}

\smallsection{Example: \ Geometric Brownian motion}
We finally consider geometric Brownian motion in dimension $d = 40$ as an example for a stochatsic delay equation just as in \eqref{eq:delay21} with a one-dimensional Brownian motion and take $\xi = (0.1, \dots, 0.1)^\top, A = -\indic + \left(a_{ij}\right)_{i,j=1}^{d}, B = \indic + \left(a_{ij}\right)_{i,j=1}^{d}, N = \indic + \left(a_{ij}\right)_{i,j=1}^{d}$, where $a_{ij} \sim \mathcal{N}(0, 10^{-4})$ is sampled i.i.d. for all $i, j \in \{1, \dots, d\}$ once at the beginning of the simulation, and $C = \operatorname{diag}(\underbrace{1, \dots, 1}_{r\,\text{times}}, \underbrace{0.01, \dots, 0.01}_{d - r \, \text{times}})$ with $r=10$. The delay time is $\tau = 0.1$, the control $u(t) = \sin(20 t) \mathbf{1}$ and the history function $f_t = 0$. Figure \ref{fig: geometric brownian motion} shows the $L^2$ errors between the full and reduced systems and compares them with the bound from Theorem \ref{theo:sdde}. We see that the error decreases if we choose $r$ big enough and that the bound is off by around two orders of magnitude.

\begin{figure}
\centering
\includegraphics[width=1.0\linewidth]{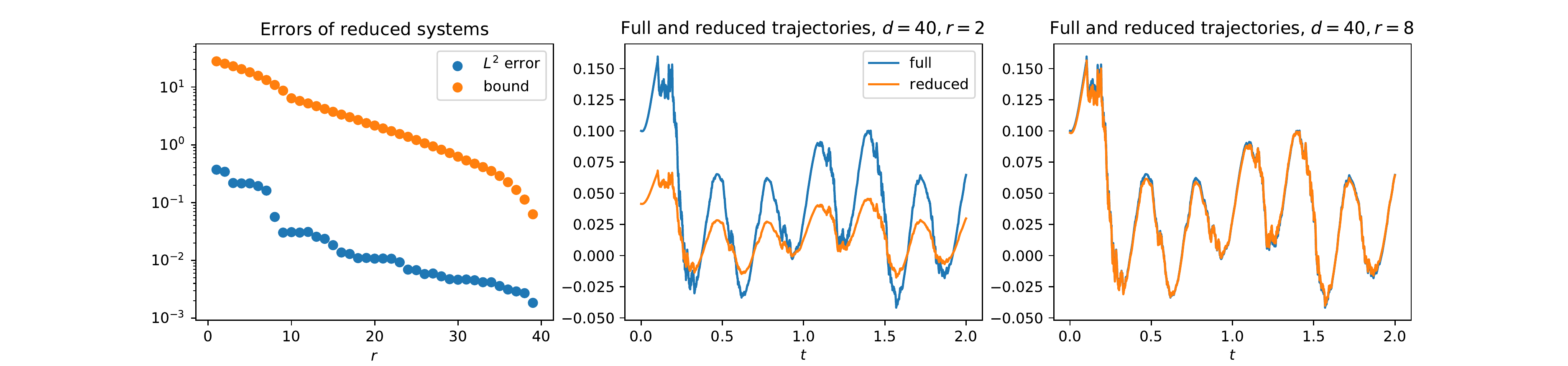}
\caption{Geometric Brownian motion with delay. Left panel: $L^2$ error and bounds for the full and reduced systems with varying dimension $r$. Right panel: one component of a trajectory of the full compared to the $r$-dimensional reduced model for different values of $r$.}
\label{fig: geometric brownian motion}
\end{figure}

\smallsection{Acknowledgements} 
This work was supported by the EPSRC grant EP/L016516/1 for the University of Cambridge CDT, the CCA (S.B.) and by Deutsche Forschungsgemeinschaft (DFG) through the grant CRC 1114 `Scaling Cascades in Complex Systems' (A05, project number 235221301). 


%



\begin{thebibliography}{0}
\bibitem[BP05]{BP} B\'{a}tkai, A. and Piazzera, S. (2005). \emph{Semigroups for Delay equations}. Research Notes in Mathematics. A K Peters/CRC Press.
\bibitem[BFS03]{BFS} B\'{a}tkai, A. and Fasanga, E. and Shvidkoy, R. (2003). \emph{Hyperbolicity of delay equations via Fourier multipliers,} Acta scientiarum mathematicarum, (69) 1-2. pp. 131-145.
\bibitem[BD11]{BD} Benner, P. and Damm, T. (2011). \emph{Lyapunov equations, energy functionals, and model order reduction of bilinear and stochastic systems.} SIAM Journal on Control and Optimization.
\bibitem[BG09]{4} Beattie, C., Gugercin, S. (2009). \emph{Interpolatory projection methods for structure-preserving model reduction.} Syst Control Lett 58(3):225-232.
\bibitem[BH19]{BH18} Becker, S. and Hartmann, C. (2018). \emph{Infinite-dimensional bilinear-and stochastic balanced truncation}.
\bibitem[BHRR20]{BHRR} Becker, S., Hartmann, C., Redmann, M., and Richter, L. (2020). \emph{ Feedback control theory and Model order reduction for stochastic equations}, \arXiv{1912.06113}.
\bibitem[BR15]{BR15} Benner, P. and Redmann, M. (2015). \emph{Model reduction for stochastic systems.} Stoch PDE: Anal Comp,Volume 3, Issue 3, pp 291-338.
\bibitem[BS17]{BS17} Butkovsky, O. and Scheutzow, M. (2017). \emph{Invariant measures for stochastic functional differential equations.} Electron. J. Probab. Volume 22, paper no. 98, 23 pp.
\bibitem[CZ95]{CZ}  Curtain R.F. and Zwart H. (1995). \emph{An introduction to infinite-dimensional linear systems theory.} Springer, New York.
\bibitem[CGP88]{CGP} Glover, K., Curtain, R., and Partington, J. (1988). \emph{Realisation and Approximation of Linear Infinite-Dimensional Systems with Error Bounds}. SIAM Journal on Control and Optimization 26:4, 863-898.
\bibitem[EN00]{EN} Engel, K-J. and Nagel, R. (2000). \emph{One-Parameter Semigroups for Linear Evolution Equations}. Springer. Graduate Texts in Mathematics.

\bibitem[GDBA19]{GDBA} Gosea, IV., Duff, IP., Benner, P., and Antoulas, AC. (2019). \emph{Model order reduction of bilinear time-delay systems}, 18th European Control Conference (ECC), 2289-2294.
\bibitem[HD11]{10} Harkort, C. and Deutscher, J. (2011). \emph{Krylov subspace methods for linear infinite-dimensional systems}. IEEE. Trans Autom Control 56(2):441-447.
\bibitem[JDM13]{JDM12} Jarlebring, E., Damm, T., and Wim, M. (2013). \emph{Model reduction of time-delay systems using position balancing and delay Lyapunov equations}. Mathematics of Control, Signals, and Systems, Volume 25, Issue 2, pp 147-166.
\bibitem[K66]{K} Kubo, R. (1966). \emph{The fluctuation-dissipation theorem}. Rep. Prog. Phys., 29(1):255.
\bibitem[LLL17]{LLL} Li, L., Liu, J.-G., and Lu, J. (2017). \emph{Fractional Stochastic Differential Equations Satisfying Fluctuation-Dissipation Theorem}, J. Stat. Phys., 169, 2, 316-339.
\bibitem[M65]{M} Mori, M. (1965). \emph{A continued-fraction representation of the time-correlation functions.}Prog. Theor. Phys.,34(3):399-416.
\bibitem[MS97]{MS97} Mao, X. and Shah, A. (1997). \emph{Exponential stability of stochastic differential
delay equations}. Stochastics: An International Journal of Probability and Stochastic
Processes, 60:1-2, 135-153
\bibitem[MP99]{19} M\"akil\"a, P. and Partington, J. (1999). \emph{Laguerre and Kautz shift approximations of delay systems.} Int J Control 72(10):932-946.
\bibitem[MP99b]{20}M\"akil\"a, P. and Partington, J. (1999). \emph{Shift operator induced approximations of delay systems.} SIAM J Control Optim 37(6):1897-1912.

\bibitem[MJM11]{24} Michiels, W., Jarlebring, E., and Meerbergen, K. (2011). \emph{Krylov-based model order reduction of time-delay
systems.} SIAM J Matrix Anal Appl 32(4):1399-1421.
\bibitem[P04]{29} Partington, J. (2004) \emph{Model reduction of delay systems.} In: Blondel V, Megretski A (eds) Unsolved problems in mathematical systems and control theory. Princeton university press, Princeton, pp 29-32.
\bibitem[PYPT10]{PYPT} Perlikowski, P., Yanchuk, S., Popovych, O., and Tass, P. (2010). \emph{Periodic patterns in a ring of delay-coupled oscillators}. Phys. Rev. E 82, 036208.
\bibitem[R19]{R1} Redmann, M. (2019).\emph{The missing link between the output and the H2-norm of bilinear systems,} \arXiv{1910.14427}.
\bibitem[RS14]{RS} Reis, T. and Selig, T. (2014). \emph{Balancing Transformations for Infinite-Dimensional Systems with Nuclear Hankel Operator}, Integr. Equ. Oper. Theory, Volume 79, Issue 1, pp 67-105.
\bibitem[SA16]{SA} Scarciotti, G. and Astolfi, A. (2016). \emph{Model Reduction of Neutral Linear and Nonlinear
Time-Invariant Time-Delay Systems With
Discrete and Distributed Delays}. IEEE Transactions on automatic control, Vol. 61, No. 6.
\bibitem[WYH10]{WYH} Wolfrum, M., Yanchuk, S., H\"ovel, P. et al. (2010). \emph{Complex dynamics in delay-differential equations with large delay}. Eur. Phys. J. Spec. Top. 191: 91.
\bibitem[ZL02]{ZL} Zhang, L. and Lam, J. (2002). On $H^2$ model reduction of bilinear systems. Automatica 38. 205-216. Pergamon.
\bibitem[ZZ12]{ZZ} Zhang, C. and Zheng, B. (2012). \emph{$Z_n$ equivariant in delay coupled dissipative Stuart-Landau oscillators}. Nonlinear Dyn 70:2359-2366.
oscillators
\end{thebibliography}
\end{document}